\documentclass[12pt,a4paper, longbibliography]{article}
\usepackage[utf8]{inputenc}
\usepackage[T1]{fontenc}
\usepackage{authblk}

\usepackage{amsmath, amscd, amsthm, amscd, amsfonts, amssymb, graphicx, color, soul, enumerate, xcolor,  mathrsfs, latexsym, bigstrut, framed, caption}
\usepackage{pstricks,pst-node,pst-tree}
\usepackage[bookmarksnumbered, colorlinks, plainpages,backref]{hyperref}
\hypersetup{colorlinks=true,linkcolor=blue, anchorcolor=green, citecolor=midblue, urlcolor=darkblue, filecolor=magenta, pdftoolbar=true}

\usepackage{tikz}
\usetikzlibrary{calc,decorations.pathreplacing,decorations.markings,positioning,shapes}

\newcommand{\sym}{\mathrm{Sym}}

\newcommand{\id}{\ensuremath{\text{\rm id}}}

\newcommand{\D}{\mathrm{D}}

\definecolor{cupgreen}{rgb}{0,0.498,0.208}
\definecolor{cupblue}{rgb}{0,0,.5}
\definecolor{cupred}{rgb}{1,0.04,0}
\definecolor{cuppink}{rgb}{0.925,0,0.545}
\definecolor{cupmagenta}{rgb}{0.624,0.161,0.424}
\definecolor{cupbrown}{rgb}{0.71,0.212,0.133}

\definecolor{cupgreen}{rgb}{0,0,0}
\definecolor{cupblue}{rgb}{0,0,0}
\definecolor{cupred}{rgb}{0,0,0}
\definecolor{cuppink}{rgb}{0,0,0}
\definecolor{cupmagenta}{rgb}{0,0,0}
\definecolor{cupbrown}{rgb}{0,0,0}

\definecolor{TITLE}{rgb}{0,0,0}
\definecolor{midblue}{rgb}{0.00,0.0,0.80}
\definecolor{darkblue}{rgb}{0.00,0.00,0.45}
\definecolor{SECTION}{rgb}{0.50,0.00,1.00}
\definecolor{THM}{rgb}{0.8,0,0.1}
\definecolor{SEC}{rgb}{0,0,1}
\newcommand{\A}{\mathcal A}
\newcommand{\B}{\mathcal B}

\newcommand{\aut}{\mathrm{Aut}}

\makeatother
\normalsize
\newtheorem{theorem}{{\color{THM} Theorem}}[section]

\DeclareRobustCommand{\stirling}{\genfrac\{\}{0pt}{}}

\newtheorem{lemma}[theorem]{{\color{THM}Lemma}}

\newtheorem{observation}[theorem]{{\color{THM}Observation}}

\theoremstyle{definition}

\newtheorem{altdefinition}[theorem]{{\color{THM}Alternate Definition\ }}

\newtheorem{example}[theorem]{{\color{THM}Example}}

\newtheorem{remark}[theorem]{{\color{THM}Remark}}
\numberwithin{equation}{section}
\newtheorem*{definition*}{\color{THM}Definition}

\date{}
\title{Distinguishing threshold of graphs}

\author[1]{Mohammad H. Shekarriz\thanks{m.shekarriz@deakin.edu.au}}

\affil[1]{School of Information Technology, Deakin University, Burwood, Australia}	

\author[2]{Bahman Ahmadi\thanks{bahman.ahmadi@shirazu.ac.ir}}

\author[2]{S. A. Talebpour\thanks{seyed.alireza.talebpour@gmail.com}}

\author[2]{M. H. Shirdareh Haghighi\thanks{shirdareh@shirazu.ac.ir}}

\affil[2]{Department of Mathematics, Shiraz University, Shiraz, Iran}

\begin{document}

	\maketitle

	\begin{abstract}
		A vertex coloring of a graph $G$ is called distinguishing if no non-identity automorphisms of $G$ can preserve it. The distinguishing number of $G$, denoted by $D(G)$, is the minimum number of colors required for such a coloring, and the distinguishing threshold of $G$, denoted by $\theta(G)$, is the minimum number $k$ such that every $k$-coloring of $G$ is distinguishing. As an alternative  definition, $\theta(G)$ is one more than the maximum number of cycles in the cycle decomposition of automorphisms of $G$. In this paper, we characterize $\theta (G)$ when $G$ is disconnected. Afterwards, we prove that, although for every positive integer $k\neq 2$ there are infinitely many graphs whose distinguishing thresholds are equal to $k$, we have $\theta(G)=2$ if and only if $\vert V(G)\vert =2$. Moreover, we show that if $\theta(G)=3$, then either $G$ is  isomorphic to one of the four graphs on~3 vertices or it is of order $2p$, where $p\neq 3,5$ is a prime number. Furthermore, we prove that $\theta(G)=D(G)$ if and only if $G$ is asymmetric, $K_n$ or $\overline{K_n}$. Finally, we consider all generalized Johnson graphs, $J(n,k,i)$, which are the graphs on all $k$-subsets of $\{1,\ldots , n\}$ where two vertices $A$ and $B$ are adjacent if  $|A\cap B|=k-i$. After studying their automorphism groups and distinguishing numbers, we calculate their distinguishing thresholds as $\theta(J(n,k,i))={n\choose k} - {n-2\choose k-1}+1$, unless $ k=\frac{n}{2}$ and $i\in\{ \frac{k}{2} , k\}$ in which case we have $\theta(J(n,k,i))={n\choose k}$.

		\medskip
		
		
	\end{abstract}

	\section{Introduction}\label{intro}
	
	An \emph{automorphism} of a graph is a symmetry of the graph, and is said to be \emph{broken} by a vertex coloring if it maps some vertex to a vertex with different color. A vertex coloring of a graph $G$ is called \emph{distinguishing} (or symmetry breaking) if it breaks all non-trivial automorphisms of $G$. The \emph{distinguishing number} of $G$, denoted by $D(G)$, is the smallest number of colors required for such a coloring. For a positive integer $d$, if there is a distinguishing coloring of $G$ with $d$ colors, we say that $G$ is \emph{$d$-distinguishable}. One can easily verify that $D(K_n ) = n$, $D(K_{n,n} ) = n + 1$, $D(P_n ) = 2$ for $n \geq 2$, $D(C_3 ) = D(C_4) = D(C_5 ) = 3$, while $D(C_n ) = 2$, for $n \geq 6$~\cite{albertson1996symmetry}.
	
	The concept has its roots back in 1970s, when Babai defined \emph{asymmetric coloring} in~\cite{Babai1977}, but it was only after publication of~\cite{albertson1996symmetry} by Albertson and Collins in 1996, that the present terminology came to be widely used. Afterwards, the subject spawned a wealth of results and newly defined graph theoretical indices. For many classes of graphs, methods to efficiently break their symmetries were already devised and, where it was feasible, general bounds on their distinguishing number were also derived.
	
	For a connected finite graph $G$, it was shown by Collins and Trenk \cite{Collins2006}, and independently by Klav\v{z}ar, Wong and Zhu \cite{klavzar2006}, that $D(G) \leq \Delta + 1$, where $\Delta$ is the largest vertex degree of~$G$. Furthermore, $D(G)=\Delta+1$ if and only if $G$ is isomorphic to $K_{\Delta+1}$, $K_{\Delta,\Delta}$, or $C_5$. 
	
	Although there are many results about finite graphs, the literature has also been enriched with numerous results on infinite graphs, see e.g.~\cite{Imrich2017Bounds,Imrich2007Infinite,Tucker2011}. The concept of distinguishing can be generalized to some other discrete structures and/or using other means of symmetry breaking. For instance, we can mention Imrich et al.~\cite{Imrich2014Endo} who considered breaking graphs' endomorphisms by coloring, Ellingham and Schroeder~\cite{Ellingham2011}  who considered symmetry breaking via partitioning and Laflamme, Nguyen Van Th\'{e} and Sauer~\cite{Laflamme2010} who considered the distinguishing number of some homogeneous structures such as directed graphs and posets.
	
	There are also several generalizations of distinguishing coloring. Collins and Trenk~\cite{Collins2006} mixed the concept with proper coloring and introduced the \emph{distinguishing chromatic number $\chi_{D}(G)$} of a graph $G$. Moreover, Kalinowski and Pil\'{s}niak~\cite{Kalinowski2015Edge} introduced the \emph{distinguishing index $D'(G)$} and the \emph{distinguishing chromatic index $\chi'_D (G)$}, while in~\cite{Kalinowski2016Total} they, along with Wo\'{z}niak,  defined and studied the analogous notions $D''(G)$ and $\chi''_D (G)$ for total coloring.
	
	The literature is also rich in results for product graphs. For example, Bogstad and Cowen~\cite{Bogstad2004} showed that for $k \geq 4$, every hypercube $Q_k$ of dimension $k$, which is the Cartesian product of $k$ copies of $K_2$, is $2$-distinguishable, while Imrich and Klav{\v{z}}ar in \cite{Imrich2006CartPower} showed that the distinguishing number of Cartesian powers of a connected graph $G$ is equal to $2$ except for $K_2^2, K_3^2, K_2^3$. Furthermore, Imrich, Jerebic and Klav\v{z}ar \cite{Imrich2008CartComp} showed that Cartesian products of relatively prime graphs whose sizes are close to each other can be distinguished with a small number of colors.
	
	The lexicographic product of two graphs $G$ and $H$, which is shown here by $G[H]$, was also a subject of symmetry breaking via vertex and edge coloring. Alikhani and Soltani in~\cite{Alikhani2018lexico} showed that, under some conditions on the automorphism group of a graph $G$, we have $D(G)\leq D(G^k)\leq D(G)+k-1$, where $G^k$ is the $k$th lexicographic power of $G$. Meanwhile, they also showed that if $G$ and $H$ are connected graphs, then  $D(H)\leq D(G[H])\leq |V(G)|\cdot D(H)$. Ahmadi, Alinaghipour and Shekarriz~\cite{ahmadi2020number} defined some indices such as $\varphi_k(G)$ (respectively, $\Phi_k(G)$), which stands for the number of non-equivalent distinguishing colorings of the graph $G$ with exactly $k$  (respectively, at most $k$) colors, and used them to refine findings on distinguishing lexicographic product. They proved that $D(X[Y])=k$ where $k$ is the least integer that $\Phi_k (Y)\geq D(X)$, when  the automorphisms of $X[Y]$ are all natural, i.~e.,  map copies of $Y$ to copies of $Y$~\cite{heminger1968}.
	
	It seems rather  easy to calculate $\Phi_k (G)$ and $\varphi_k (G)$ when $G$ is a path or a cycle. However, the calculations are not easy in the general case, and known algorithms have exponential running times. Consequently, it might take a very long time for a computer algebra system to count the number of non-equivalent distinguishing colorings of a graph $G$ with non-trivial  symmetries on $n$ vertices, even when $n$ is  as small as $10$. However, when $k$ is large enough, even for the case of  large graphs,  the calculations are much easier. This motivated Ahmadi, Alinaghipour and Shekarriz~\cite{ahmadi2020number} to define the  following.
	
	\begin{definition*}\label{Def:theta} \textnormal{\cite{ahmadi2020number}}
		For a graph $G$, the distinguishing threshold $\theta(G)$ is the minimum number $k$ of colors such that any coloring of the graph $G$ with $k$ colors is distinguishing.
	\end{definition*}
	
	Clearly we have $D(G)\leq \theta(G)\leq |V(G)|$. It is not also hard to see that $\theta(K_n)=\theta(\overline{K_n})=n$, $\theta(K_{m,n})=m+n$, $\theta(P_n)=\lceil\frac{n}{2}\rceil+1$, for $n\geq 2$, and $\theta(C_n)=\lfloor\frac{n}{2}\rfloor+2$, for $n\geq 3$~\cite{ahmadi2020number}. The main objective of this paper is to continue the study of this parameter. 
	
	An alternative definition can be given in terms of $\aut(G)$. Let $\alpha$ be a non-identity automorphism of $G$, and let $c(\alpha)$ be the number of cycles of $\alpha$ as a permutation (fixed points count as cycles) and set $c(\id)=0$. To distinguish $\alpha$, at least one non-trivial cycle must be assigned two different colors. By the pigeonhole principle we have the following alternative definition.
	
	\begin{altdefinition}\label{max-lem}
		Given a graph $G$, its distinguishing threshold is 
		$$\theta(G) = 1 + \max\left\{c(\alpha) \; :\; \alpha\in\aut(G)\right\}.$$
	\end{altdefinition}
	In particular, we can talk about the threshold of any permutation group. As
	far as we know, this seems to be a new invariant for permutation groups.
	
	Suppose that $G$ has some non-trivial symmetries and assume that $\alpha$ is a non-identity automorphism of $G$ for which $c(\alpha)$ is maximum. Then $\alpha$ has prime order since otherwise $\alpha^m$ has more cycles than $\alpha$ for any $m > 1$ properly dividing the order of $\alpha$. Consequently, we have the following observation.
	
	\begin{observation}\label{prime-alpha}
		Let $G$ be a graph with $\aut(G)\neq \{\id\}$ and let $\alpha\in \aut(G)$ be such that $\theta(G)=c(\alpha)+1$. Then there is a prime number $p$ such that $o(\alpha)=p$. In particular, the length of every cycle of $\alpha$ is either $p$ or $1$.
	\end{observation}
	
	We start our study by presenting some preliminaries in Section~\ref{sec:pre}. Then we thoroughly look at the distinguishing threshold in Section~\ref{threshold}, where we first consider $\theta(G)$ when $G$ is a disconnected graph. Then we study graphs with small thresholds by showing that, while all asymmetric graphs have distinguishing threshold equal to~1, the only graphs with $\theta(G) = 2$ are $K_2$ and $\overline{K_2}$. Moreover, a graph $G$ with $\theta(G)=3$ is either one of the four graphs on~3 vertices or a bi-regular graph on~$2p$ vertices where $p\neq 3,5$ is a prime number. In addition, we prove that $\theta(G)=D(G)$ if and only if $G$ is asymmetric, $K_n$ or $\overline{K_n}$. Furthermore, it is shown that the distinguishing threshold of an infinite graph is either~1 or infinity. 
	
	The relation between the distinguishing threshold and the cycle structure of the automorphism group motivates us to consider the graphs whose automorphism groups have been studied thoroughly. A rich family of such examples are the graphs that come from \emph{association schemes}, particularly, the \emph{Johnson scheme}. For detailed studies on this subject, we refer the reader  to the texts such as~\cite{eiichi1984algebraic} by Eiichi and Tatsuro. The interesting fact about the graphs in the Johnson scheme, i.e. the \emph{generalized Johnson graphs}, is that any permutation on $\{1,\ldots, n\}$ induces an automorphism of all the generalized Johnson graphs on the same set of vertices; that is, $\sym(n)$ is isomorphic to a subgroup of their automorphism groups. Jones~\cite{jones2005automorphisms} obtained the automorphism group of any \emph{merged Johnson graph} from which the automorphism group of any generalized Johnson graph can be obtained. Furthermore, Jones' result enabled Kim, Kwon and Lee~\cite{kim2011distinguishing} to obtain the distinguishing number of any merged Johnson graph which, in turn, leads to the distinguishing number of any generalized Johnson graph.  
	
	An important sub-family of generalized Johnson graphs, is the family of \textit{Kneser graphs} $K(n,k)$, to which the Petersen graph belongs. In~\cite{ahmadi2020number}, the authors have also addressed the problem of determining the distinguishing threshold of the Kneser graphs, in a special case.  More specifically,  for $n\geq 5$, they have proved that $\theta(K(n,2))=\frac{1}{2}(n^2-3n+6)$. In Section~\ref{johnson}, we continue this study and, along with listing the distinguishing number of all generalized Johnson graphs, we compute their distinguishing threshold.

	All graphs in this paper are assumed to be simple (undirected and loopless) and finite, unless otherwise stated. We use standard notation in graph theory which can be found in~\cite{diestel2017} by Diestel.

	\section{Preliminaries}\label{sec:pre}

	Two colorings $c_1$ and $c_2$ of a graph $G$ are called \textit{equivalent} if there is an automorphism $\alpha$ of $G$ such that $c_1(v) = c_2(\alpha(v))$, for all $v \in V(G)$. The number of non-equivalent distinguishing colorings of a graph $G$ with $\{1, \cdots, k\}$ as the set of admissible colors is shown by $\Phi_k(G)$, while the number of non-equivalent $k$-distinguishing colorings of a graph $G$ with $\{1, \cdots, k\}$ as the set of colors is shown by $\varphi_k(G)$~\cite{ahmadi2020number}. These two indices are related as 
	
	\[
	\Phi_k(G) = \sum_{i=D(G)}^k {k\choose i} \varphi_i(G).
	\]
	
	\noindent By simple counting arguments, one observes that $\Phi_k(P_n) = \frac{1}{2}(k^n - k^{\lceil \frac{n}{2} \rceil})$ when $n,k \geq 1$, and $\Phi_k(K_n) = {k \choose n}$ when $n\geq 2$ and $k \geq n$~\cite{ahmadi2020number}. 
	
	Calculating $\varphi_k(G)$ in general cases, requires counting the number of $k$-distinguishing colorings of $G$, while  when $k$ is large enough so that every $k$-coloring of $G$ is distinguishing, i.~e. when $k\geq\theta(G)$,  we have 
	
	\[
	\varphi_k (G)=\frac{k! \stirling{n}{k}}{|\aut(G)|},
	\] 
	
	\noindent where $\stirling{n}{k}$ stands for the Stirling number of the second kind~\cite{ahmadi2020number}.
	
	For an automorphism $\alpha$ and a vertex $v\in V(G)$, the ordered tuple  $$\sigma=(v, \alpha(v), \alpha^2 (v),\ldots,\alpha^{r-1} (v))$$ forms a \emph{cycle} of length $r$ provided that $r$ is the least integer such that $\alpha^{r}(v)=v$. We define the \emph{base} of the cycle $\sigma$ to be $\B(\sigma)=\{v, \alpha(v), \alpha^2 (v),\ldots,\alpha^{r-1} (v)\}$. The number of cycles of an automorphism $\alpha$, as also noted in the introduction, is denoted here by $c(\alpha)$.
	
	Using the cycle structures  of the elements of the automorphism group of a graph, one can obtain the exact value of the distinguishing threshold, as we saw in Definition~\ref{max-lem}.
	
	In the next section, we use the notion of circulant graphs, which are briefly discussed here. For any group $(\Gamma, \cdot)$ and any non-empty subset $S\subseteq \Gamma$ which is closed under inversion and $1_\Gamma\notin S$, the \emph{Cayley graph} $\mathrm{Cay}(\Gamma,S)$ on $\Gamma$ with the \emph{connection set} $S$ is defined to be the graph whose vertices are the elements of $\Gamma$, and in which two vertices $g,h$ are adjacent if $g\cdot h^{-1}\in S$. Cayley graphs are clearly  vertex-transitive. In the special case where $\Gamma=(\mathbb{Z}_n, +)$ and $S$ is a symmetric subset of $\mathbb{Z}_n^* = \mathbb{Z}_n\setminus\{0\}$,  $G=\mathrm{Cay}(\mathbb{Z}_n , S)$ is called a \emph{circulant graph}, see~\cite{morris2006} by Morris or~\cite{godsil-royle} by Godsil and Royle. 
	An alternative  way of describing a circulant graph $G$ is that there is an automorphism $\alpha \in \aut(G)$ such that $c(\alpha)=1$; that is, $G$ is a circulant graph if and only if it has an automorphism which contains all the vertices of $G$ in one cycle.
	
	Suppose that $G=\mathrm{Cay}(A,S)$ for an abelian group $(A,+)$. Since $S$ is closed under inversion, the inversion map $inv(a) = -a$ is obviously an automorphism of $G$. If $A$ is cyclic (that is, $G$ is circulant) of order $m>2$, then $\langle A, inv \rangle = A \rtimes C_2$ is the dihedral group $D_{2m}$, so $D_{2m}$ is a subgroup of $\aut(G)$, see~\cite[Proof of Corollary 1]{Kagno1946} by Kagno. Therefore, we have the following lemma.
	
	\begin{lemma}\label{dihedral-graph}
		Let $G$ be a circulant graph of odd order $m>2$. Then $D_{2m}\leq \aut(G)$.
	\end{lemma}
	
	A graph $G$ is called \emph{bi-regular} if there are two (not necessarily distinct) positive integers $d_1$ and $d_2$ such that the  degree of any vertex $v\in G$ is either $d_1$ or $d_2$. We denote the set of all vertices with vertex degrees $d_1$ and $d_2$ by $\mathcal{A}_1$ and $\mathcal{A}_2$, respectively. We use this notation in Theorem~\ref{theta3}.
	
	\section{Distinguishing threshold}\label{threshold}

	In what follows, we consider the distinguishing threshold for general graphs and prove some results which can provide useful machinery for this paper as well as future studies. Our default assumption is that all graphs here are connected; one exception is the following theorem in which we consider the distinguishing threshold for disconnected graphs.
	
	In order to state the theorem, we will make use of the following notation. Suppose that $G$ is a graph with connected components $G_1, \ldots, G_k$, where all the $G_i$ are asymmetric. Then we consider the isomorphism congruence classes $\mathcal{C}_1,\ldots, \mathcal{C}_m$ of the graphs $G_1,\ldots, G_k$, where we assume that the $\mathcal{C}_j$'s are increasingly ordered in the sense that if $j<\ell$, and $G_{i_j}\in \mathcal{C}_j$ and $G_{i_\ell}\in \mathcal{C}_\ell$, then $|V(G_{i_j})|\leq |V(G_{i_\ell})|$. We define $\nu(G)$ to be $|V(G_{i_s})|$, where $G_{i_s}\in \mathcal{C}_s$ and~$s$ is the smallest integer with the property that $|\mathcal{C}_s|>1$; and if there is no such $s$, then we define $\nu(G)$ to be $|V(G)|$. Note that, for example, if $k=1$ (i.e. if $G$ is a connected asymmetric graph), then $\nu(G)=|V(G)|$.
	
	\begin{theorem}\label{union}
		Let $G_1, G_2, \ldots, G_k$ be arbitrary connected   graphs and let $G = \cup_{i=1}^{k} G_i$.  
		\begin{enumerate}[(a)]
			\item If for every $1 \leq i \leq k$, $\aut(G_i) \ne \{\mathrm{id}\}$,  then 
			\[\theta(G) = \max_{1\leq i\leq k} \; \left\{\theta(G_i) + \sum_{j \neq i}|V(G_j)| \right\}.\]
			
			\item  If $\aut(G_i) = \{\mathrm{id}\}$, for all $1 \leq i \leq k$, then $\theta(G)=|V(G)|-\nu(G)+1$.
			
			\item If $\{A,B\}$ is a non-trivial partition of $\{1,\ldots,k\}$, $\aut(G_i) \ne  \{\mathrm{id}\}$, for $i\in A$, and $\aut(G_i) = \{\mathrm{id}\}$, for  $i \in B$, then set  $G_A = \bigcup_{i \in A} G_i$ and $G_B = \bigcup_{i \in B} G_i$. In this case, we have
			\[\theta(G) = \max \{ \theta(G_A) + |V(G_B)|,\; \theta(G_B) + |V(G_A)| \},\]
			unless   $G_B$ is asymmetric and $\theta(G_A) + |V(G_B)|\leq  \theta(G_B) + |V(G_A)|$, in which case  we have $\theta(G) = \theta(G_A) + |V(G_B)|$.
		\end{enumerate}
	\end{theorem}
	\begin{proof}
		To prove   (a), without loss of generality, we assume
		\[q = \max_{1\leq i\leq k} \; \left\{\theta(G_i) + \sum_{j \neq i}|V(G_j)| \right\} = \theta(G_1) + \sum_{j = 2}^k|V(G_j)|.\]
		It is obvious that  $ \theta(G)\geq q$. Thus, it  is enough to show that any arbitrary  coloring  of  $G$ with $q$ colors  is distinguishing. Suppose on contrary that $c$ is a non-distinguishing coloring of $G$ with $q$ colors. Then, we may assume that the number of colors which are used in $G_1$ is $\theta(G_1) + r$, where $0 < r \leq |V(G_1)| - \theta(G_1)$, and that the components  $G_{i_1}, \ldots ,G_{i_m}$ receive $\vert V(G_{i_1})\vert-t_1, \ldots, \vert V(G_{i_m})\vert -t_m$  colors, respectively, where $0<t_j < |V(G_{i_j})|$, for $1 \leq j \leq m$, and $t_1 + t_2 + \cdots + t_m = r$. Since $c$ is not distinguishing, we assume, without loss of generality, that $|V(G_{i_1})| - t_1 < \theta(G_{i_1})$. Hence,
		\[|V(G_{i_1})| + |V(G_{i_2})|+ \cdots + |V(G_{i_m})| - (t_1 + \cdots + t_m) < \theta(G_{i_1}) + |V(G_{i_2})| + \cdots + |V(G_{i_m})|;\]
		thus, since $r = t_1 + \cdots + t_m$ and $r \leq |V(G_1)| - \theta(G_1)$, we have 
		\[ |V(G_{i_1})| + |V(G_{i_2})|+ \cdots + |V(G_{i_m})| + \theta(G_1) < \theta(G_{i_1}) + |V(G_{i_2})| + \cdots + |V(G_{i_m})| + |V(G_1)|,\]
		which contradicts the definition of $q$.
		
		To show part (b), we note that if the $G_i$s are mutually non-isomorphic, then by the definition, we have $\nu(G)=|V(G)|$ and the result follows because $G$ is asymmetric. Otherwise, let $\nu(G)=|V(H)|$. Clearly, any coloring with $|V(G)|-|V(H)| +1$ colors  breaks the symmetry of $G$; however, there is an automorphism $\beta$  which  interchanges $H$ with another component and fixes all other vertices. The automorphism $\beta$     has $\vert V(G)\vert-2\vert V(H)\vert$ singleton cycles and $\vert V(H)\vert$ cycles of size 2. Hence  we have $c(\beta)=|V(G)|-|V(H)|$ which completes the proof.
		
		Finally, to prove  (c), we first  assume that  
		\[q = \max \{ \theta(G_A) + |V(G_B)|,\; \theta(G_B) + |V(G_A)| \} = \theta(G_A) + |V(G_B)|.
		\]
		It is easy to see that $ \theta(G)\geq q$. On the other hand, suppose that there exists a non-distinguishing coloring $c$ with $q$ colors. Similar to the proof  of  (a), assume that $G_A$ and  $G_B$ receive $\theta(G_A) + r$ and $|V(G_B)| - r$ colors,  respectively, where  $0< r \leq |V(G_A)| - \theta(G_A)$. Since $c$ is not distinguishing, we have 
		$|V(G_B)| - r < \theta(G_B)$ which implies that 
		\[ |V(G_B)| - |V(G_A)| + \theta(G_A) < \theta(G_B).\]
		Hence, 
		\[|V(G_B)| + \theta(G_A) <  \theta(G_B)+|V(G_A)|,\]
		which is a contradiction and shows that $\theta(G) \leq q$.
		
		Now, assume $q= \theta(G_B) + |V(G_A)|$. If $\theta(G_B) = 1$, then it is obvious that $\theta(G) = \theta(G_A) + |V(G_B)|$ and the case where $\theta(G_B) > 1$ follows using a similar argument as in part (a). 
	\end{proof}

	In the following lemma, we make use of the notion of \textit{motion}, which was introduced by Russell and Sundaram in~\cite{Russell}. The motion $m(\alpha)$ of an automorphism $\alpha\in \aut(G)$ is the number of vertices of $G$ that are not fixed by $\alpha$, and the motion of $G$ is defined as 
	$$m(G) = \min\left\{m(\alpha) \;:\;\alpha \in \aut(G)\setminus \{\id \} \right\}.$$

	\begin{lemma}\label{motion}
		Let $G$ be a graph on $n$ vertices. Then, we have \[\theta(G)\geq n-m(G)+2.\]
	\end{lemma}
	\begin{proof}
		Let $\beta\in\aut(G)$ with $m(\beta)=m(G)$ and consider a coloring with $n-m(G)+1$ colors which assigns $n-m(G)$ colors to the fixed vertices of $\beta$ and another color to all the  vertices moved by $\beta$. This coloring is not distinguishing and, therefore, the statement follows.
	\end{proof}
	
	Since for an asymmetric graph $G$, we have $\theta(G) = 1$~\cite{ahmadi2020number}, there are infinitely many graphs whose distinguishing threshold is $1$. Therefore, it is a natural question to ask whether, for an integer $k\geq 2$, there are finitely or infinitely many graphs whose distinguishing threshold is $k$. We first consider the case $k=2$ in the following theorem which states that there are only two such graphs. 
	
	\begin{theorem}\label{theta2}
		The only graphs with the distinguishing threshold $2$ are $K_2$ and $\overline{K_2}$.
	\end{theorem}
	\begin{proof}
		The only   graphs on 2 vertices are $K_2$ and its complement $\overline{K_2}$. Suppose $G$ is a graph on more than 2 vertices for which we have $\theta (G)=2$. Then, $G$ must have some non-trivial symmetries and any automorphism $\alpha\neq \mathrm{id}$ must have at most one cycle of length $n= |V(G)|$, which implies that $G$ is circulant. Moreover, by Observation~\ref{prime-alpha},~$n$ must be a prime number greater than 2. Therefore, by Lemma~\ref{dihedral-graph} we must have $D_{2n}\leq \aut(G)$ which means that there is an automorphism $\gamma\in\aut(G)$ which has $\lfloor \frac{n}{2} \rfloor+1>1$ cycles. This contradiction completes the proof.
	\end{proof}

	Characterizing the graphs whose distinguishing threshold is~$3$ seems to be more complicated. There are infinitely many such graphs, but except for some small ones, all such graphs are of a special order.

	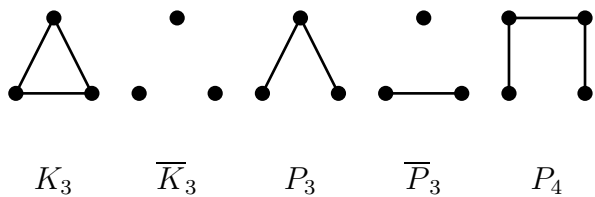
\begin{figure} [h]
		\centering
		\begin{center}
			\begin{tabular}{ c c c c c }
				\begin{tikzpicture}  [scale=0.5]
					
					\tikzstyle{every path}=[line width=1pt]
					
					\newdimen\ms
					\ms=0.1cm
					\tikzstyle{s1}=[color=black,fill,rectangle,inner sep=3]
					\tikzstyle{c1}=[color=black,fill,circle,inner sep={\ms/8},minimum size=2*\ms]
					

					\coordinate (a1) at  (1,2);
					\coordinate (a3) at (2,0);
					\coordinate (a5) at (0,0);

					
					\draw [color=black] (a1) -- (a3);
					\draw [color=black] (a3) -- (a5);
					\draw [color=black] (a5) -- (a1);

					
					\draw (a1) coordinate[c1];
					\draw (a3) coordinate[c1];
					\draw (a5) coordinate[c1];
					
				\end{tikzpicture}
				&
				\begin{tikzpicture}  [scale=0.5]
					
					\tikzstyle{every path}=[line width=1pt]
					
					\newdimen\ms
					\ms=0.1cm
					\tikzstyle{s1}=[color=black,fill,rectangle,inner sep=3]
					\tikzstyle{c1}=[color=black,fill,circle,inner sep={\ms/8},minimum size=2*\ms]
					

					\coordinate (a1) at  (1,2);
					\coordinate (a3) at (2,0);
					\coordinate (a5) at (0,0);
					
					

					
					\draw (a1) coordinate[c1];
					\draw (a3) coordinate[c1];
					\draw (a5) coordinate[c1];

				\end{tikzpicture}
				&
				\begin{tikzpicture}  [scale=0.5]
					
					\tikzstyle{every path}=[line width=1pt]
					
					\newdimen\ms
					\ms=0.1cm
					\tikzstyle{s1}=[color=black,fill,rectangle,inner sep=3]
					\tikzstyle{c1}=[color=black,fill,circle,inner sep={\ms/8},minimum size=2*\ms]
					

					\coordinate (a1) at  (1,2);
					\coordinate (a3) at (2,0);
					\coordinate (a5) at (0,0);
					
					
					\draw [color=black] (a1) -- (a3);
					\draw [color=black] (a5) -- (a1);

					
					\draw (a1) coordinate[c1];
					\draw (a3) coordinate[c1];;
					\draw (a5) coordinate[c1];

				\end{tikzpicture}
				&
				\begin{tikzpicture}  [scale=0.5]
					
					\tikzstyle{every path}=[line width=1pt]
					
					\newdimen\ms
					\ms=0.1cm
					\tikzstyle{s1}=[color=black,fill,rectangle,inner sep=3]
					\tikzstyle{c1}=[color=black,fill,circle,inner sep={\ms/8},minimum size=2*\ms]
					

					\coordinate (a1) at  (1,2);
					\coordinate (a3) at (2,0);
					\coordinate (a5) at (0,0);


					\draw [color=black] (a3) -- (a5);

					
					\draw (a1) coordinate[c1];
					\draw (a3) coordinate[c1];
					\draw (a5) coordinate[c1];
					
				\end{tikzpicture}
				&
				\begin{tikzpicture}  [scale=0.5]
					
					\tikzstyle{every path}=[line width=1pt]
					
					\newdimen\ms
					\ms=0.1cm
					\tikzstyle{s1}=[color=black,fill,rectangle,inner sep=3]
					\tikzstyle{c1}=[color=black,fill,circle,inner sep={\ms/8},minimum size=2*\ms]
					

					\coordinate (a1) at  (0,2);
					\coordinate (a2) at (2,0);
					\coordinate (a3) at (0,0);
					\coordinate (a4) at (2,2);

					
					\draw [color=black] (a1) -- (a3);
					\draw [color=black] (a4) -- (a1);
					\draw [color=black] (a4) -- (a2);

					
					\draw (a1) coordinate[c1];
					\draw (a2) coordinate[c1];
					\draw (a3) coordinate[c1];
					\draw (a4) coordinate[c1];
					
				\end{tikzpicture}
				\\ \vspace{1mm} \\
				$K_3$&$\overline{K}_3$&$P_3$&$\overline{P}_3$&$P_4$\\
			\end{tabular}

		\end{center}
		\caption{small graphs whose distinguishing threshold equal to 3}
		\label{fig:theta=3}
	\end{figure}

	\begin{theorem}\label{theta3}
		Let $G$ be a graph on $n$ vertices for which we have $\theta(G)=3$. Then, either
		\begin{itemize}
			\item[(a)] $n=3$, or
			\item[(b)] $n=2p$ where $p\neq 3,5$ is a prime number, $G$ is a connected bi-regular graph with $\mathcal{A}_1 =\{v_0,\ldots,v_{p-1}\}$ and $\mathcal{A}_2 =\{u_0,\ldots,u_{p-1} \}$, and the induced subgraphs $G[\mathcal{A}_1]$ and $G[\mathcal{A}_2]$ are non-isomorphic circulant graphs.
		\end{itemize}
	\end{theorem}
	\begin{proof}
		If $n=3$ then there is nothing to prove. We know that there are eleven graphs on 4 vertices, among which only $P_4$ has the distinguishing threshold 3 (this is indeed easy to check directly or by reasoning), see Figure~\ref{fig:theta=3}. It is straightforward to check that $P_4$ satisfies item (b). Thus we suppose that $G$ is a graph with $\theta (G)=3$ and $n\geq 5$. We prove that the  statement (b) holds for $G$.
		
		Lemma~\ref{motion} implies that $m(G)\geq n-1$. This means that for an arbitrary $\alpha\in \aut (G)\setminus \{\id \}$, we have $m(\alpha)\geq n-1$ and $c(\alpha)\leq 2$. Thus, we have the following two cases.
		
		\begin{itemize}
			\item[Case 1.] $m(\alpha)=n-1$. Since $m(\alpha)\neq n$, it must be the case that $c(\alpha) =2$. Consequently, the only vertex that is fixed by $\alpha$, say $v$, must belong to a separate cycle from the rest of vertices. Thus, either $v$ is on an edge, which means that it has to be adjacent to all other vertices, or $v$ is an isolated vertex. Both cases imply that $\alpha\vert_{G\setminus v}$ is an automorphism for $G\setminus v$. This means that $\theta(G\setminus v)=2$ because if $\theta(G\setminus v) = 3$, then one 3-coloring would involve coloring $v$ with its own color, and $G \setminus v$ with only two colors. Since $v$ is either universal or independent, automorphisms of $G \setminus v$ induce	automorphisms of $G$. So if we choose the 2-coloring of $G \setminus v$ to be non-distinguishing, this 3-coloring of $G$ will also be non-distinguishing, which contradicts our assumption. Therefore, by Theorem~\ref{theta2},  $G\setminus v$ must be isomorphic to either $K_2$ or $\overline{K_2}$. This is a contradiction to the fact that $n\geq 5$.
			
			\item[Case 2.] $m(\alpha)=n$. We split this case, in turn, into the following two sub-cases.
			\begin{itemize}
				\item[Case 2.1.] $\alpha$ is mono-cyclic. Since $n\geq 5$, by Lemma~\ref{dihedral-graph} there is an automorphism $\gamma$ which has at least 3 cycles; this implies that $\theta(G)\geq 4$, a contradiction.
				
				\item[Case 2.2.] $c(\alpha)=2$. Since $\theta(G)=3$, the number of cycles of $\alpha$ is maximum. Thus, by Observation~\ref{prime-alpha} both cycles of $\alpha$ have an equal prime length, say $p$.
			\end{itemize}
		\end{itemize}
		Therefore, we have $c(\alpha)=2$ and the two cycles of $\alpha$ have the same length which is a prime number $p\geq 3$. Without loss of generality we can assume that $\alpha=(v_0,\ldots,v_{p-1})(u_0,\ldots,u_{p-1})$. We set $\mathcal{A}_1 =\{v_0\ldots,v_{p-1}\}$ and $\mathcal{A}_2 =\{u_0,\ldots, u_{p-1}\}$. Since $\mathcal{A}_1\cup \mathcal{A}_2 =V(G)$, $G$ must be bi-regular because the degrees of the vertices in $\mathcal{A}_1$ have to be the equal; the same for the vertices in $\mathcal{A}_2$. Moreover, since $\alpha$ has two cycles, both induced sub-graphs $G[\mathcal{A}_1]$ and $G[\mathcal{A}_2 ]$ are circulant graphs of prime order $p$.
		
		Now, suppose that we have $G[\mathcal{A}_1 ]\cong G[\mathcal{A}_2 ]$. Let $u_i \sim v_j$ for some $i,j=0,\ldots , p-1$. Then, using the automorphism $\alpha^{-i-j}$, we must have $u_{-j}\sim v_{-i}$. This shows that the map
		
		\[\eta: V(G) \longrightarrow V(G)\]
		\[\eta(x)=\left\{ \begin{gathered}
			u_{-i}\hspace{13mm}x=v_i\in\mathcal{A}_1 \\
			v_{-i} \hspace{13mm} x=u_i\in\mathcal{A}_2
		\end{gathered}\right.\]
		is another automorphism of $G$.  Since $n\geq 5$ and $\eta$ has $p\geq 3$ cycles, it means that $\theta(G)\geq 4$, a contradiction. 
		
		Note that there are two circulant graphs on~3 vertices, namely $K_3$ and $\overline{K_3}$, and three circulant graphs on~5 vertices, say $K_5$, $\overline{K_5}$ and $C_5$. Since any combination of non-isomorphic pairs from these graphs as $G[\A_1]$ and $G[\A_2]$ cannot generate a graph $G$ whose distinguishing threshold is~3, we have $p\neq 3,5$. Figures~\ref{fig:k3k3^} and~\ref{fig:k5^c5} illustrate this fact, which completes the proof.
	\end{proof}

	\begin{figure} [h]
		\centering
		\begin{center}
			\begin{tabular}{ c c c c }
				\begin{tikzpicture}  [scale=0.2]
					
					\tikzstyle{every path}=[line width=1pt]
					
					\newdimen\ms
					\ms=0.1cm
					\tikzstyle{s1}=[color=black,fill,rectangle,inner sep=3]
					\tikzstyle{c1}=[color=black,fill,circle,inner sep={\ms/8},minimum size=2*\ms]
					

					\coordinate (a1) at  (0,2);
					\coordinate (a2) at (-1.7323,-0.99989);
					\coordinate (a3) at (1.7323,-0.99989);
					\coordinate (a4) at  (0,5);
					\coordinate (a5) at (-4.3302,-2.4997);
					\coordinate (a6) at (4.3302,-2.4497);

					
					\draw [color=black] (a1) -- (a2);
					\draw [color=black] (a2) -- (a3);
					\draw [color=black] (a3) -- (a1);

					
					\draw (a1) coordinate[c1];
					\draw (a2) coordinate[c1];
					\draw (a3) coordinate[c1];
					\draw (a4) coordinate[c1];
					\draw (a5) coordinate[c1];
					\draw (a6) coordinate[c1];
					
				\end{tikzpicture}
				&
				\begin{tikzpicture}  [scale=0.2]
					
					\tikzstyle{every path}=[line width=1pt]
					
					\newdimen\ms
					\ms=0.1cm
					\tikzstyle{s1}=[color=black,fill,rectangle,inner sep=3]
					\tikzstyle{c1}=[color=black,fill,circle,inner sep={\ms/8},minimum size=2*\ms]
					

					\coordinate (a1) at  (0,2);
					\coordinate (a2) at (-1.7323,-0.99989);
					\coordinate (a3) at (1.7323,-0.99989);
					\coordinate (a4) at  (0,5);
					\coordinate (a5) at (-4.3302,-2.4997);
					\coordinate (a6) at (4.3302,-2.4497);

					
					\draw [color=black] (a1) -- (a2);
					\draw [color=black] (a2) -- (a3);
					\draw [color=black] (a3) -- (a1);
					\draw [color=black] (a1) -- (a4);
					\draw [color=black] (a2) -- (a5);
					\draw [color=black] (a3) -- (a6);

					
					\draw (a1) coordinate[c1];
					\draw (a2) coordinate[c1];
					\draw (a3) coordinate[c1];
					\draw (a4) coordinate[c1];
					\draw (a5) coordinate[c1];
					\draw (a6) coordinate[c1];
					
				\end{tikzpicture}
				&
				\begin{tikzpicture}  [scale=0.2]
					
					\tikzstyle{every path}=[line width=1pt]
					
					\newdimen\ms
					\ms=0.1cm
					\tikzstyle{s1}=[color=black,fill,rectangle,inner sep=3]
					\tikzstyle{c1}=[color=black,fill,circle,inner sep={\ms/8},minimum size=2*\ms]
					

					\coordinate (a1) at  (0,2);
					\coordinate (a2) at (-1.7323,-0.99989);
					\coordinate (a3) at (1.7323,-0.99989);
					\coordinate (a4) at  (0,-5);
					\coordinate (a5) at (4.3302,2.4997);
					\coordinate (a6) at (-4.3302,2.4497);

					
					\draw [color=black] (a1) -- (a2);
					\draw [color=black] (a2) -- (a3);
					\draw [color=black] (a3) -- (a1);
					\draw [color=black] (a3) -- (a4);
					\draw [color=black] (a1) -- (a5);
					\draw [color=black] (a2) -- (a6);
					\draw [color=black] (a2) -- (a4);
					\draw [color=black] (a3) -- (a5);
					\draw [color=black] (a1) -- (a6);

					
					\draw (a1) coordinate[c1];
					\draw (a2) coordinate[c1];
					\draw (a3) coordinate[c1];
					\draw (a4) coordinate[c1];
					\draw (a5) coordinate[c1];
					\draw (a6) coordinate[c1];

				\end{tikzpicture}
				&
				\begin{tikzpicture}  [scale=0.2]
					
					\tikzstyle{every path}=[line width=1pt]
					
					\newdimen\ms
					\ms=0.1cm
					\tikzstyle{s1}=[color=black,fill,rectangle,inner sep=3]
					\tikzstyle{c1}=[color=black,fill,circle,inner sep={\ms/8},minimum size=2*\ms]
					

					\coordinate (a1) at  (0,2);
					\coordinate (a2) at (-1.7323,-0.99989);
					\coordinate (a3) at (1.7323,-0.99989);
					\coordinate (a4) at  (0,5);
					\coordinate (a5) at (-4.3302,-2.4997);
					\coordinate (a6) at (4.3302,-2.4497);

					
					\draw [color=black] (a1) -- (a2);
					\draw [color=black] (a2) -- (a3);
					\draw [color=black] (a3) -- (a1);
					\draw [color=black] (a1) -- (a4);
					\draw [color=black] (a2) -- (a5);
					\draw [color=black] (a3) -- (a6);
					\draw [color=black] (a2) -- (a4);
					\draw [color=black] (a3) -- (a5);
					\draw [color=black] (a1) -- (a6);
					\draw [color=black] (a3) -- (a4);
					\draw [color=black] (a1) -- (a5);
					\draw [color=black] (a2) -- (a6);

					
					\draw (a1) coordinate[c1];
					\draw (a2) coordinate[c1];
					\draw (a3) coordinate[c1];
					\draw (a4) coordinate[c1];
					\draw (a5) coordinate[c1];
					\draw (a6) coordinate[c1];
					
				\end{tikzpicture}

				\\ \vspace{1mm} \\
				$\theta=6$&$\theta=5$&$\theta=5$&$\theta=6$\\
			\end{tabular}

		\end{center}
		\caption{four different graphs for the proof of Theorem~\ref{theta3} for the case  which $G[\A_1]=K_3$ and $G[\A_2]=\overline{K_3}$}
		\label{fig:k3k3^}
	\end{figure}
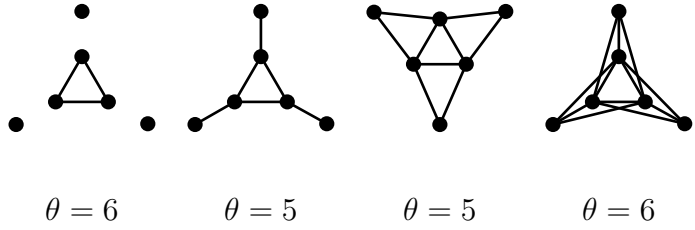

	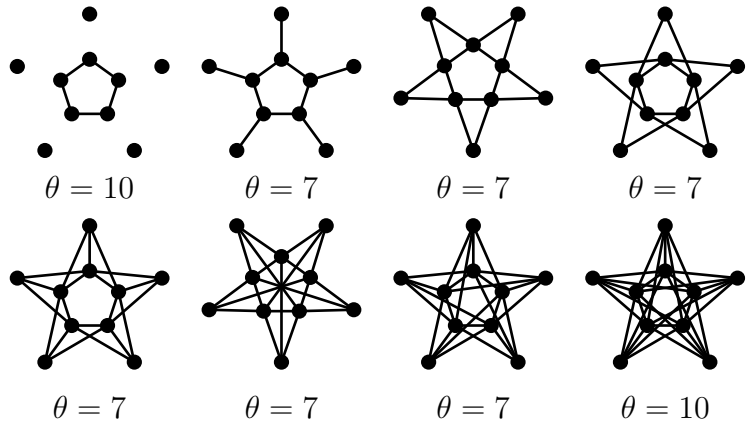
\begin{figure} [h]
		\centering
		\begin{center}
			\begin{tabular}{ c c c c }
				\begin{tikzpicture}  [scale=0.2]
					
					\tikzstyle{every path}=[line width=1pt]
					
					\newdimen\ms
					\ms=0.1cm
					\tikzstyle{s1}=[color=black,fill,rectangle,inner sep=3]
					\tikzstyle{c1}=[color=black,fill,circle,inner sep={\ms/8},minimum size=2*\ms]
					

					\coordinate (a1) at  (0,2);
					\coordinate (a2) at (-1.90209,0.6181);
					\coordinate (a3) at (-1.17569,-1.61794);
					\coordinate (a4) at  (1.17569,-1.61794);
					\coordinate (a5) at (1.90209,0.6181);
					\coordinate (a6) at (0,5);
					\coordinate (a7) at  (-4.75522,1.52546);
					\coordinate (a8) at (-2.93922,-4.04486);
					\coordinate (a9) at (2.93922,-4.04486);
					\coordinate (a10) at  (4.75522,1.52546);

					
					\draw [color=black] (a1) -- (a2);
					\draw [color=black] (a2) -- (a3);
					\draw [color=black] (a3) -- (a4);
					\draw [color=black] (a4) -- (a5);
					\draw [color=black] (a5) -- (a1);

					
					\draw (a1) coordinate[c1];
					\draw (a2) coordinate[c1];
					\draw (a3) coordinate[c1];
					\draw (a4) coordinate[c1];
					\draw (a5) coordinate[c1];
					\draw (a6) coordinate[c1];
					\draw (a7) coordinate[c1];
					\draw (a8) coordinate[c1];
					\draw (a9) coordinate[c1];
					\draw (a10) coordinate[c1];
					
				\end{tikzpicture}
				&
				\begin{tikzpicture}  [scale=0.2]
					
					\tikzstyle{every path}=[line width=1pt]
					
					\newdimen\ms
					\ms=0.1cm
					\tikzstyle{s1}=[color=black,fill,rectangle,inner sep=3]
					\tikzstyle{c1}=[color=black,fill,circle,inner sep={\ms/8},minimum size=2*\ms]
					

					\coordinate (a1) at  (0,2);
					\coordinate (a2) at (-1.90209,0.6181);
					\coordinate (a3) at (-1.17569,-1.61794);
					\coordinate (a4) at  (1.17569,-1.61794);
					\coordinate (a5) at (1.90209,0.6181);
					\coordinate (a6) at (0,5);
					\coordinate (a7) at  (-4.75522,1.52546);
					\coordinate (a8) at (-2.93922,-4.04486);
					\coordinate (a9) at (2.93922,-4.04486);
					\coordinate (a10) at  (4.75522,1.52546);

					
					\draw [color=black] (a1) -- (a2);
					\draw [color=black] (a2) -- (a3);
					\draw [color=black] (a3) -- (a4);
					\draw [color=black] (a4) -- (a5);
					\draw [color=black] (a5) -- (a1);
					\draw [color=black] (a1) -- (a6);
					\draw [color=black] (a2) -- (a7);
					\draw [color=black] (a3) -- (a8);
					\draw [color=black] (a4) -- (a9);
					\draw [color=black] (a5) -- (a10);

					
					\draw (a1) coordinate[c1];
					\draw (a2) coordinate[c1];
					\draw (a3) coordinate[c1];
					\draw (a4) coordinate[c1];
					\draw (a5) coordinate[c1];
					\draw (a6) coordinate[c1];
					\draw (a7) coordinate[c1];
					\draw (a8) coordinate[c1];
					\draw (a9) coordinate[c1];
					\draw (a10) coordinate[c1];
					
				\end{tikzpicture}
				&
				\begin{tikzpicture}  [scale=0.2]
					
					\tikzstyle{every path}=[line width=1pt]
					
					\newdimen\ms
					\ms=0.1cm
					\tikzstyle{s1}=[color=black,fill,rectangle,inner sep=3]
					\tikzstyle{c1}=[color=black,fill,circle,inner sep={\ms/8},minimum size=2*\ms]
					
					
					\coordinate (a1) at  (0,2);
					\coordinate (a2) at (-1.90209,0.6181);
					\coordinate (a3) at (-1.17569,-1.61794);
					\coordinate (a4) at  (1.17569,-1.61794);
					\coordinate (a5) at (1.90209,0.6181);
					\coordinate (a6) at (0,-5);
					\coordinate (a7) at  (4.75522,-1.52546);
					\coordinate (a8) at (2.93922,4.04486);
					\coordinate (a9) at (-2.93922,4.04486);
					\coordinate (a10) at  (-4.75522,-1.52546);

					
					\draw [color=black] (a1) -- (a2);
					\draw [color=black] (a2) -- (a3);
					\draw [color=black] (a3) -- (a4);
					\draw [color=black] (a4) -- (a5);
					\draw [color=black] (a5) -- (a1);
					\draw [color=black] (a3) -- (a6);
					\draw [color=black] (a4) -- (a7);
					\draw [color=black] (a5) -- (a8);
					\draw [color=black] (a1) -- (a9);
					\draw [color=black] (a2) -- (a10);
					\draw [color=black] (a4) -- (a6);
					\draw [color=black] (a5) -- (a7);
					\draw [color=black] (a1) -- (a8);
					\draw [color=black] (a2) -- (a9);
					\draw [color=black] (a3) -- (a10);

					
					\draw (a1) coordinate[c1];
					\draw (a2) coordinate[c1];
					\draw (a3) coordinate[c1];
					\draw (a4) coordinate[c1];
					\draw (a5) coordinate[c1];
					\draw (a6) coordinate[c1];
					\draw (a7) coordinate[c1];
					\draw (a8) coordinate[c1];
					\draw (a9) coordinate[c1];
					\draw (a10) coordinate[c1];
					
				\end{tikzpicture}
				&
				\begin{tikzpicture}  [scale=0.2]
					
					\tikzstyle{every path}=[line width=1pt]
					
					\newdimen\ms
					\ms=0.1cm
					\tikzstyle{s1}=[color=black,fill,rectangle,inner sep=3]
					\tikzstyle{c1}=[color=black,fill,circle,inner sep={\ms/8},minimum size=2*\ms]
					

					\coordinate (a1) at  (0,2);
					\coordinate (a2) at (-1.90209,0.6181);
					\coordinate (a3) at (-1.17569,-1.61794);
					\coordinate (a4) at  (1.17569,-1.61794);
					\coordinate (a5) at (1.90209,0.6181);
					\coordinate (a6) at (0,5);
					\coordinate (a7) at  (-4.75522,1.52546);
					\coordinate (a8) at (-2.93922,-4.04486);
					\coordinate (a9) at (2.93922,-4.04486);
					\coordinate (a10) at  (4.75522,1.52546);

					
					\draw [color=black] (a1) -- (a2);
					\draw [color=black] (a2) -- (a3);
					\draw [color=black] (a3) -- (a4);
					\draw [color=black] (a4) -- (a5);
					\draw [color=black] (a5) -- (a1);
					\draw [color=black] (a5) -- (a6);
					\draw [color=black] (a1) -- (a7);
					\draw [color=black] (a2) -- (a8);
					\draw [color=black] (a3) -- (a9);
					\draw [color=black] (a4) -- (a10);
					\draw [color=black] (a2) -- (a6);
					\draw [color=black] (a3) -- (a7);
					\draw [color=black] (a4) -- (a8);
					\draw [color=black] (a5) -- (a9);
					\draw [color=black] (a1) -- (a10);

					
					\draw (a1) coordinate[c1];
					\draw (a2) coordinate[c1];
					\draw (a3) coordinate[c1];
					\draw (a4) coordinate[c1];
					\draw (a5) coordinate[c1];
					\draw (a6) coordinate[c1];
					\draw (a7) coordinate[c1];
					\draw (a8) coordinate[c1];
					\draw (a9) coordinate[c1];
					\draw (a10) coordinate[c1];
					
				\end{tikzpicture}

				\\ \vspace{1mm}
				$\theta=10$&$\theta=7$&$\theta=7$&$\theta=7$\\
				\vspace{1mm} 

				\begin{tikzpicture}  [scale=0.2]
					
					\tikzstyle{every path}=[line width=1pt]
					
					\newdimen\ms
					\ms=0.1cm
					\tikzstyle{s1}=[color=black,fill,rectangle,inner sep=3]
					\tikzstyle{c1}=[color=black,fill,circle,inner sep={\ms/8},minimum size=2*\ms]
					
					
					\coordinate (a1) at  (0,2);
					\coordinate (a2) at (-1.90209,0.6181);
					\coordinate (a3) at (-1.17569,-1.61794);
					\coordinate (a4) at  (1.17569,-1.61794);
					\coordinate (a5) at (1.90209,0.6181);
					\coordinate (a6) at (0,5);
					\coordinate (a7) at  (-4.75522,1.52546);
					\coordinate (a8) at (-2.93922,-4.04486);
					\coordinate (a9) at (2.93922,-4.04486);
					\coordinate (a10) at  (4.75522,1.52546);

					
					\draw [color=black] (a1) -- (a2);
					\draw [color=black] (a2) -- (a3);
					\draw [color=black] (a3) -- (a4);
					\draw [color=black] (a4) -- (a5);
					\draw [color=black] (a5) -- (a1);
					\draw [color=black] (a5) -- (a6);
					\draw [color=black] (a1) -- (a7);
					\draw [color=black] (a2) -- (a8);
					\draw [color=black] (a3) -- (a9);
					\draw [color=black] (a4) -- (a10);
					\draw [color=black] (a2) -- (a6);
					\draw [color=black] (a3) -- (a7);
					\draw [color=black] (a4) -- (a8);
					\draw [color=black] (a5) -- (a9);
					\draw [color=black] (a1) -- (a10);
					\draw [color=black] (a2) -- (a7);
					\draw [color=black] (a3) -- (a8);
					\draw [color=black] (a4) -- (a9);
					\draw [color=black] (a5) -- (a10);
					\draw [color=black] (a1) -- (a6);

					
					\draw (a1) coordinate[c1];
					\draw (a2) coordinate[c1];
					\draw (a3) coordinate[c1];
					\draw (a4) coordinate[c1];
					\draw (a5) coordinate[c1];
					\draw (a6) coordinate[c1];
					\draw (a7) coordinate[c1];
					\draw (a8) coordinate[c1];
					\draw (a9) coordinate[c1];
					\draw (a10) coordinate[c1];
					
				\end{tikzpicture}
				&
				\begin{tikzpicture}  [scale=0.2]
					
					\tikzstyle{every path}=[line width=1pt]
					
					\newdimen\ms
					\ms=0.1cm
					\tikzstyle{s1}=[color=black,fill,rectangle,inner sep=3]
					\tikzstyle{c1}=[color=black,fill,circle,inner sep={\ms/8},minimum size=2*\ms]
					
					
					\coordinate (a1) at  (0,2);
					\coordinate (a2) at (-1.90209,0.6181);
					\coordinate (a3) at (-1.17569,-1.61794);
					\coordinate (a4) at  (1.17569,-1.61794);
					\coordinate (a5) at (1.90209,0.6181);
					\coordinate (a6) at (0,-5);
					\coordinate (a7) at  (4.75522,-1.52546);
					\coordinate (a8) at (2.93922,4.04486);
					\coordinate (a9) at (-2.93922,4.04486);
					\coordinate (a10) at  (-4.75522,-1.52546);

					
					\draw [color=black] (a1) -- (a2);
					\draw [color=black] (a2) -- (a3);
					\draw [color=black] (a3) -- (a4);
					\draw [color=black] (a4) -- (a5);
					\draw [color=black] (a5) -- (a1);
					\draw [color=black] (a4) -- (a6);
					\draw [color=black] (a5) -- (a7);
					\draw [color=black] (a3) -- (a8);
					\draw [color=black] (a4) -- (a9);
					\draw [color=black] (a5) -- (a10);
					\draw [color=black] (a3) -- (a6);
					\draw [color=black] (a4) -- (a7);
					\draw [color=black] (a5) -- (a8);
					\draw [color=black] (a1) -- (a9);
					\draw [color=black] (a2) -- (a10);
					\draw [color=black] (a2) -- (a7);
					\draw [color=black] (a1) -- (a8);
					\draw [color=black] (a2) -- (a9);
					\draw [color=black] (a3) -- (a10);
					\draw [color=black] (a1) -- (a6);

					
					\draw (a1) coordinate[c1];
					\draw (a2) coordinate[c1];
					\draw (a3) coordinate[c1];
					\draw (a4) coordinate[c1];
					\draw (a5) coordinate[c1];
					\draw (a6) coordinate[c1];
					\draw (a7) coordinate[c1];
					\draw (a8) coordinate[c1];
					\draw (a9) coordinate[c1];
					\draw (a10) coordinate[c1];
				\end{tikzpicture}
				&
				\begin{tikzpicture}  [scale=0.2]
					
					\tikzstyle{every path}=[line width=1pt]
					
					\newdimen\ms
					\ms=0.1cm
					\tikzstyle{s1}=[color=black,fill,rectangle,inner sep=3]
					\tikzstyle{c1}=[color=black,fill,circle,inner sep={\ms/8},minimum size=2*\ms]
					
					
					\coordinate (a1) at  (0,2);
					\coordinate (a2) at (-1.90209,0.6181);
					\coordinate (a3) at (-1.17569,-1.61794);
					\coordinate (a4) at  (1.17569,-1.61794);
					\coordinate (a5) at (1.90209,0.6181);
					\coordinate (a6) at (0,5);
					\coordinate (a7) at  (-4.75522,1.52546);
					\coordinate (a8) at (-2.93922,-4.04486);
					\coordinate (a9) at (2.93922,-4.04486);
					\coordinate (a10) at  (4.75522,1.52546);

					
					\draw [color=black] (a1) -- (a2);
					\draw [color=black] (a2) -- (a3);
					\draw [color=black] (a3) -- (a4);
					\draw [color=black] (a4) -- (a5);
					\draw [color=black] (a5) -- (a1);
					\draw [color=black] (a5) -- (a6);
					\draw [color=black] (a1) -- (a7);
					\draw [color=black] (a2) -- (a8);
					\draw [color=black] (a3) -- (a9);
					\draw [color=black] (a4) -- (a10);
					\draw [color=black] (a2) -- (a6);
					\draw [color=black] (a3) -- (a7);
					\draw [color=black] (a4) -- (a8);
					\draw [color=black] (a5) -- (a9);
					\draw [color=black] (a1) -- (a10);
					\draw [color=black] (a2) -- (a7);
					\draw [color=black] (a3) -- (a8);
					\draw [color=black] (a4) -- (a9);
					\draw [color=black] (a5) -- (a10);
					\draw [color=black] (a1) -- (a6);
					\draw [color=black] (a6) -- (a3);
					\draw [color=black] (a7) -- (a4);
					\draw [color=black] (a8) -- (a5);
					\draw [color=black] (a9) -- (a1);
					\draw [color=black] (a10) -- (a2);

					
					\draw (a1) coordinate[c1];
					\draw (a2) coordinate[c1];
					\draw (a3) coordinate[c1];
					\draw (a4) coordinate[c1];
					\draw (a5) coordinate[c1];
					\draw (a6) coordinate[c1];
					\draw (a7) coordinate[c1];
					\draw (a8) coordinate[c1];
					\draw (a9) coordinate[c1];
					\draw (a10) coordinate[c1];
					
				\end{tikzpicture}
				&
				\begin{tikzpicture}  [scale=0.2]
					
					\tikzstyle{every path}=[line width=1pt]
					
					\newdimen\ms
					\ms=0.1cm
					\tikzstyle{s1}=[color=black,fill,rectangle,inner sep=3]
					\tikzstyle{c1}=[color=black,fill,circle,inner sep={\ms/8},minimum size=2*\ms]
					
					
					\coordinate (a1) at  (0,2);
					\coordinate (a2) at (-1.90209,0.6181);
					\coordinate (a3) at (-1.17569,-1.61794);
					\coordinate (a4) at  (1.17569,-1.61794);
					\coordinate (a5) at (1.90209,0.6181);
					\coordinate (a6) at (0,5);
					\coordinate (a7) at  (-4.75522,1.52546);
					\coordinate (a8) at (-2.93922,-4.04486);
					\coordinate (a9) at (2.93922,-4.04486);
					\coordinate (a10) at  (4.75522,1.52546);

					
					\draw [color=black] (a1) -- (a2);
					\draw [color=black] (a2) -- (a3);
					\draw [color=black] (a3) -- (a4);
					\draw [color=black] (a4) -- (a5);
					\draw [color=black] (a5) -- (a1);
					\draw [color=black] (a5) -- (a6);
					\draw [color=black] (a1) -- (a7);
					\draw [color=black] (a2) -- (a8);
					\draw [color=black] (a3) -- (a9);
					\draw [color=black] (a4) -- (a10);
					\draw [color=black] (a2) -- (a6);
					\draw [color=black] (a3) -- (a7);
					\draw [color=black] (a4) -- (a8);
					\draw [color=black] (a5) -- (a9);
					\draw [color=black] (a1) -- (a10);
					\draw [color=black] (a2) -- (a7);
					\draw [color=black] (a3) -- (a8);
					\draw [color=black] (a4) -- (a9);
					\draw [color=black] (a5) -- (a10);
					\draw [color=black] (a1) -- (a6);
					\draw [color=black] (a6) -- (a3);
					\draw [color=black] (a7) -- (a4);
					\draw [color=black] (a8) -- (a5);
					\draw [color=black] (a9) -- (a1);
					\draw [color=black] (a10) -- (a2);
					\draw [color=black] (a6) -- (a4);
					\draw [color=black] (a7) -- (a5);
					\draw [color=black] (a8) -- (a1);
					\draw [color=black] (a9) -- (a2);
					\draw [color=black] (a10) -- (a3);

					
					\draw (a1) coordinate[c1];
					\draw (a2) coordinate[c1];
					\draw (a3) coordinate[c1];
					\draw (a4) coordinate[c1];
					\draw (a5) coordinate[c1];
					\draw (a6) coordinate[c1];
					\draw (a7) coordinate[c1];
					\draw (a8) coordinate[c1];
					\draw (a9) coordinate[c1];
					\draw (a10) coordinate[c1];

				\end{tikzpicture}

				\\ \vspace{1mm} 
				$\theta=7$&$\theta=7$&$\theta=7$&$\theta=10$

			\end{tabular}

		\end{center}
		\caption{eight different  graphs for the proof of Theorem~\ref{theta3} for the case which $G[\A_1]=C_5$ and $G[\A_2]=\overline{K_5}$; the case which $G[\A_1]=C_5$ and $G[\A_2]=K_5$ is   similar. When $G[\A_1]=K_5$ and $G[\A_2]=\overline{K_5}$ the situation is also similar, but there are only five different  such graphs.}
		\label{fig:k5^c5}
	\end{figure}

	According to Theorem~\ref{theta3}, there might exist numerous graphs whose distinguishing thresholds are equal to~ $3$. Indeed, we can construct an infinite family of such graphs. The following example demonstrate how this can be done.
	
	\begin{example}
		It is easy to show that every non-trivial automorphism of the graph $G$ in Figure~\ref{theta=3} has only two cycles: one whose base is  the $7$ outer vertices and the other whose base is  the $7$ inner ones. Moreover, it is not hard to show that $\aut (G) \cong \mathbb{Z}_7$. 
		In a similar way and by replacing $7$ by any prime number $p>7$, one can construct graphs with the similar property. Therefore, there are infinitely many graphs with the distinguishing threshold~$3$.

		\begin{figure}[h!]
			\centering
			\begin{tikzpicture}  [scale=0.6]
				
				\tikzstyle{every path}=[line width=1pt]
				
				\newdimen\ms
				\ms=0.1cm
				\tikzstyle{s1}=[color=black,fill,rectangle,inner sep=3]
				\tikzstyle{c1}=[color=black,fill,circle,inner sep={\ms/8},minimum size=2*\ms]
				

				\coordinate (a1) at  (0,2);
				\coordinate (a2) at (1.5636,1.2447);
				\coordinate (a3) at (1.9499,-0.4449);
				\coordinate (a4) at (0.8679,-1.8019);
				\coordinate (a5) at (-0.8676,-1.8020);
				\coordinate (a6) at (-1.9498,-0.4453);
				\coordinate (a7) at (-1.5639,1.2447);
				\coordinate (a8) at (0,4);
				\coordinate (a9) at (3.1272,2.4940);
				\coordinate (a10) at (3.8997,-0.8899);
				\coordinate (a11) at (1.7358,-3.6037);
				\coordinate (a12) at (-1.7351,-3.6040);
				\coordinate (a13) at (-3.8996,-0.8906);
				\coordinate (a14) at (-3.1277,2.3945);
				
				\draw [color=black] (a1) -- (a2);
				\draw [color=black] (a2) -- (a3);
				\draw [color=black] (a3) -- (a4);
				\draw [color=black] (a4) -- (a5);
				\draw [color=black] (a5) -- (a6);
				\draw [color=black] (a6) -- (a7);
				\draw [color=black] (a1) -- (a7);
				\draw [color=black] (a8) -- (a1);
				\draw [color=black] (a8) -- (a2);
				\draw [color=black] (a8) -- (a4);
				\draw [color=black] (a9) -- (a2);
				\draw [color=black] (a9) -- (a3);
				\draw [color=black] (a9) -- (a5);
				\draw [color=black] (a10) -- (a3);
				\draw [color=black] (a10) -- (a4);
				\draw [color=black] (a10) -- (a6);
				\draw [color=black] (a11) -- (a4);
				\draw [color=black] (a11) -- (a5);
				\draw [color=black] (a11) -- (a7);
				\draw [color=black] (a12) -- (a5);
				\draw [color=black] (a12) -- (a6);
				\draw [color=black] (a12) -- (a1);
				\draw [color=black] (a13) -- (a6);
				\draw [color=black] (a13) -- (a7);
				\draw [color=black] (a13) -- (a2);
				\draw [color=black] (a14) -- (a7);
				\draw [color=black] (a14) -- (a1);
				\draw [color=black] (a14) -- (a3);
				
				\draw (a1) coordinate[c1];
				\draw (a2) coordinate[c1];
				\draw (a3) coordinate[c1];
				\draw (a4) coordinate[c1];
				\draw (a5) coordinate[c1];
				\draw (a6) coordinate[c1];
				\draw (a7) coordinate[c1];
				\draw (a8) coordinate[c1];
				\draw (a9) coordinate[c1];
				\draw (a10) coordinate[c1];
				\draw (a11) coordinate[c1];
				\draw (a12) coordinate[c1];
				\draw (a13) coordinate[c1];
				\draw (a14) coordinate[c1];
			\end{tikzpicture}
			\caption{a graph on 14 vertices whose distinguishing threshold equals 3}
			\label{theta=3}
		\end{figure}
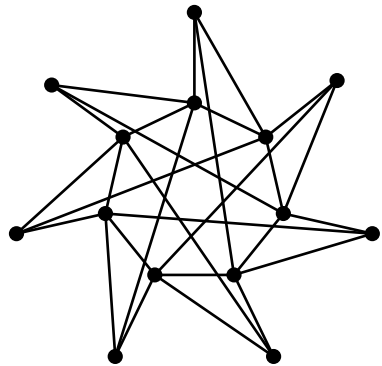
	\end{example}
	
	Furthermore, in the following example,  we extend this argument to all thresholds  $t\geq 4$.
	
	\begin{example}
		The graphs $G_1$  and $G_2$ in Figures~\ref{fig:theta=4} and~\ref{fig:theta=5}, respectively, have the property that their automorphism group is isomorphic to $\mathbb{Z}_p$ where $p\geq 3$ is a prime number, and we have $\theta(G_1)=4$ and $\theta(G_2)=5$. Similar to $G_1$, one can construct infinitely many graphs with threshold equal to $4$. Also, the graph $G_2$ can easily be used to generate a graph whose distinguishing threshold is $t\geq 5$. If we make the central vertex $o$ adjacent with an end vertex of a path of length $t$ to generate a new graph $G_{t+5}$, then we have $\theta (G_{t+5})=t+5$.
	\end{example}

	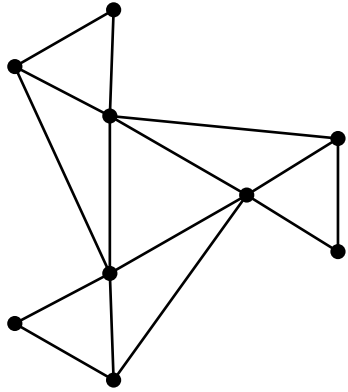
\begin{figure}[h!]
		\centering
		\begin{center}
			
			\begin{tikzpicture}  [scale=0.3]
				
				\tikzstyle{every path}=[line width=1pt]
				
				\newdimen\ms
				\ms=0.1cm
				\tikzstyle{s1}=[color=black,fill,rectangle,inner sep=3]
				\tikzstyle{c1}=[color=black,fill,circle,inner sep={\ms/8},minimum size=2*\ms]
				
				
				\coordinate (a1) at  (8,-2.5);
				\coordinate (a3) at (4,0);
				\coordinate (a5) at (8,2.5);
				\coordinate (a6) at (-1.9998,3.4942);
				\coordinate (a7) at (-1.8343,8.1873);
				\coordinate (a8) at (-6.1648,5.6785);
				\coordinate (a9) at (-2.0004,-3.4638);
				\coordinate (a10) at (-6.1658,-5.6773);
				\coordinate (a11) at (-1.8358,-8.178);
				
				\draw [color=black] (a1) -- (a3);
				\draw [color=black] (a3) -- (a5);
				\draw [color=black] (a5) -- (a1);
				\draw [color=black] (a6) -- (a5);
				\draw [color=black] (a6) -- (a3);
				\draw [color=black] (a6) -- (a7);
				\draw [color=black] (a8) -- (a7);
				\draw [color=black] (a6) -- (a8);
				\draw [color=black] (a9) -- (a8);
				\draw [color=black] (a6) -- (a9);
				\draw [color=black] (a9) -- (a3);
				\draw [color=black] (a9) -- (a10);
				\draw [color=black] (a9) -- (a11);
				\draw [color=black] (a11) -- (a10);
				\draw [color=black] (a3) -- (a11);
				
				\draw (a1) coordinate[c1];
				\draw (a3) coordinate[c1];
				\draw (a5) coordinate[c1];
				\draw (a6) coordinate[c1];
				\draw (a7) coordinate[c1];
				\draw (a8) coordinate[c1];
				\draw (a9) coordinate[c1];
				\draw (a10) coordinate[c1];
				\draw (a11) coordinate[c1];
			\end{tikzpicture}
			
		\end{center}
		\caption{graph $G_1$ whose distinguishing threshold is~ 4}
		\label{fig:theta=4}
	\end{figure}

	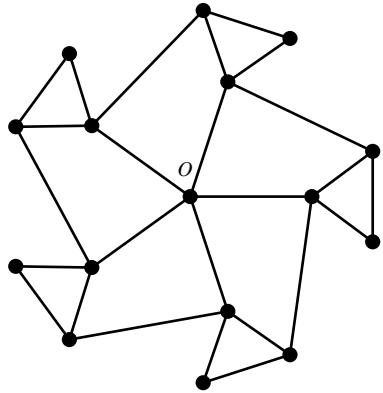
\begin{figure}[h!]
		\centering
		\begin{tikzpicture}  [scale=0.4]
			
			\tikzstyle{every path}=[line width=1pt]
			
			\newdimen\ms
			\ms=0.1cm
			\tikzstyle{s1}=[color=black,fill,rectangle,inner sep=3]
			\tikzstyle{c1}=[color=black,fill,circle,inner sep={\ms/8},minimum size=2*\ms]
			

			\coordinate (a1) at  (6,-1.5);
			\coordinate (a3) at (4,0);
			\coordinate (a5) at (6,1.5);
			\coordinate (a6) at (1.2362,-3.8044);
			\coordinate (a7) at (3.28,-5.2432);
			\coordinate (a8) at (0.4265,-6.1699);
			\coordinate (a9) at (0,0);
			\coordinate (a10) at (1.2362,3.8041);
			\coordinate (a11) at (3.2810,5.2426);
			\coordinate (a12) at (0.4276,6.1698);
			\coordinate (a13) at (-3.2358,2.3513);
			\coordinate (a14) at (-3.9720,4.7406);
			\coordinate (a15) at (-5.7356,2.3135);
			\coordinate (a16) at (-3.2358,-2.3513);
			\coordinate (a17) at (-5.7361,-2.3124);
			\coordinate (a18) at (-3.9728,-4.7399);
			
			\draw [color=black] (a1) -- (a3);
			\draw [color=black] (a3) -- (a5);
			\draw [color=black] (a5) -- (a1);
			\draw [color=black] (a6) -- (a7);
			\draw [color=black] (a6) -- (a8);
			\draw [color=black] (a8) -- (a7);
			\draw [color=black] (a3) -- (a7);
			\draw [color=black] (a9) -- (a6);
			\draw [color=black] (a9) -- (a3);
			\draw [color=black] (a10) -- (a9);
			\draw [color=black] (a10) -- (a5);
			\draw [color=black] (a10) -- (a11);
			\draw [color=black] (a12) -- (a11);
			\draw [color=black] (a10) -- (a12);
			\draw [color=black] (a9) -- (a13);
			\draw [color=black] (a12) -- (a13);
			\draw [color=black] (a14) -- (a13);
			\draw [color=black] (a15) -- (a13);
			\draw [color=black] (a14) -- (a15);
			\draw [color=black] (a16) -- (a15);
			\draw [color=black] (a16) -- (a9);
			\draw [color=black] (a17) -- (a16);
			\draw [color=black] (a18) -- (a16);
			\draw [color=black] (a18) -- (a17);
			\draw [color=black] (a18) -- (a6);
			
			\draw (a1) coordinate[c1];
			\draw (a3) coordinate[c1];
			\draw (a5) coordinate[c1];
			\draw (a6) coordinate[c1];
			\draw (a7) coordinate[c1];
			\draw (a8) coordinate[c1];
			\draw (a9) coordinate[c1,label=above:$o\textnormal{ }$];
			\draw (a10) coordinate[c1];
			\draw (a11) coordinate[c1];
			\draw (a12) coordinate[c1];
			\draw (a13) coordinate[c1];
			\draw (a14) coordinate[c1];
			\draw (a15) coordinate[c1];
			\draw (a16) coordinate[c1];
			\draw (a17) coordinate[c1];
			\draw (a18) coordinate[c1];
		\end{tikzpicture}

		\caption{graph $G_2$ whose distinguishing threshold is 5}
		\label{fig:theta=5}
	\end{figure}

	We recall that for any graph $G$, the threshold is bounded $D(G)\leq \theta(G)\leq |V(G)|$, where the second bound follows from the fact that any coloring which  uses as many colors as $|V(G)|$ is trivially distinguishing. In what follows, we study the graphs for which the lower bound holds with equality. More specifically, we prove that the lower bound holds with equality only if the threshold is trivial unless the graph is asymmetric.

	\begin{theorem}\label{D=theta}
		If for a graph $G$, we have $\theta(G)=D(G)$, then either $\theta(G)=1$ or $\theta(G)=|V(G)|$.
	\end{theorem}
	\begin{proof}
		Suppose that $G$ has some symmetries and on the contrary we have $\theta(G)<|V(G)|$. Let $\theta(G)=q$ and assume $\alpha$ is a non-identity automorphism of $G$ for which $c(\alpha)$ is  maximum and that $\alpha=\sigma_1\cdots\sigma_r\gamma_1\cdots\gamma_s$ is the cycle decomposition of $\alpha$, where $\sigma_i$s are of lengths $\geq 2$ and all $\gamma_j$s have the length~$1$. According to Definition~\ref{max-lem}, we have $r+s=q-1$. 
		As $c(\alpha)$ is maximum, by Observation~\ref{prime-alpha} we know that the lengths of $\sigma_i$s are all equal to $p_\alpha$, where $p_\alpha$ is a prime number. Since $q=D(G)$, every coloring of $G$ with $q-1$ colors is non-distinguishing. We consider two cases.
		
		\begin{itemize}
			\item[Case 1.] $r\geq 2$. We split this case into the following two sub-cases.
			
			\begin{itemize}
				\item[Case 1.1] $p_\alpha\geq 3$. In this case, we consider the following coloring $c$ with $q-1$ colors. We color  all the vertices in $\B(\sigma_i)$, $i=2,\ldots,r$, with the unique color $i$, and the vertex in $\B(\gamma_j)$, $j=1,\ldots,s$, with the color $r+j$; then, color a certain vertex $v\in \B(\sigma_1)$ with color~$2$ and all the other vertices in $\B(\sigma_1)$ with color~$1$.  Since $c$ is not distinguishing, there must be a non-identity automorphism $\beta$ of $G$ which preserves $c$. This implies that 
				$\beta(\B(\gamma_j))=\B(\gamma_j)$, for $j=1,\ldots,s$, $\beta(\B(\sigma_i))=\B(\sigma_i)$, for  $i=3,\ldots,r$, $\beta(\B(\sigma_1)-\{v\})=\B(\sigma_1)-\{v\}$ and $\beta(\B(\sigma_2)\cup\{v\})=\B(\sigma_2)\cup\{v\}$. This shows that there is a power $\beta^\ell$ of $\beta$ which is not the identity and $c(\beta^\ell)>c(\alpha)$, a contradiction. 
				
				\item[Case 1.2] $p_\alpha=2$. By employing the method described in Case 1.1, we obtain an automorphism $\beta$ for which the prime number $p_\beta$ is bigger than $2$, while we have $c(\beta) =c(\alpha)$. However, the number of cycles of $\beta$ whose length is greater than or equal to 2 is strictly less than $r$. Consequently, since in this case the number of cycles of length 1 of $\beta$ is never less than 1, this case either leads to Case 1.1 or Case 2.1, both of which result in contradictions. 
			\end{itemize}

			\item[Case 2.] $r=1$. We split this case into the following two sub-cases.
			
			\begin{itemize}
				\item[Case 2.1.] $s\geq 1$. Since $\theta(G)<|V(G)|$, the length of $\sigma_1$ is at least~$3$, because if the length of $\sigma_1$ is 2 then we must have $c(\alpha) = \vert V(G)\vert-1$. If its length is equal to~$3$, and $s=1$, then $G$ is a graph on 4 vertices which has to be either $K_4$, $\overline{K_4}$, $K_{1,3}$ or  $K_1\cup K_3$. But this is a contradiction because the distinguishing thresholds for all these graphs are equal to 4, their number of vertices. On the other hand, if the length of $\sigma_1$ is equal to~$3$ and $s\geq 2$ or if the length of $\sigma_1$ is  at least~$5$,  then   we consider the following coloring $c$ with $q-1$ colors. We color    the vertex  in $\B(\gamma_j)$, $j=1,\ldots,s$, with the color $j$; then, color a certain vertex $v\in \B(\sigma_1)$ with color~$1$ and  the other two vertices in $\B(\sigma_1)$ with color~$s+1$. Then, since $c$ is not distinguishing, the existence of a non-identity automorphism $\beta$ of $G$, leads us to a similar contradiction as in Case~1.1. 
				
				\item[Case 2.2.]  $s=0$.  In this case we have $\alpha=\sigma_1$. In other words, the maximum number of cycles in non-identity automorphisms is 1 which implies that $\theta(G)=2$. Now the assumption that $\theta(G)<|V(G)|$, contradicts Theorem~\ref{theta2}. 
			\end{itemize}
		\end{itemize}
		This completes the proof.
	\end{proof}
	
	We note that the converse of Theorem~\ref{D=theta} does not hold; for example $\theta(K_{n,n})=|V(K_{n,n})|$ while $D(K_{n,n})=n+1$ (see~\cite{ahmadi2020number}). Furthermore, in the light of Theorem~\ref{D=theta}, we can rephrase Theorem~\ref{theta2} as follows: for any graph $G$,  $\theta(G)=D(G)=2$ holds if and only if $G$ is either $K_2$ or $\overline{K_2}$. It turns out that we can generalize this result in the following fashion.
	
	\begin{theorem}\label{D=theta_then_G=K_n_or_complement} 
		For a graph $G$ on $n$ vertices, we have $\theta(G)=D(G)$ if and only if $G$ is either asymmetric, the complete graph $K_n$ or the empty graph $\overline{K_n}$.
	\end{theorem}
	\begin{proof}
		The ``if'' part is obvious. For the converse, suppose that $\aut(G)$ is not trivial. Note that according to Theorem~\ref{D=theta}, any coloring of $G$ with $n-1$ colors is not distinguishing. Consider two distinct  vertices $u,v\in V(G)$ and color both of them with color~$1$ and assign a unique color $2,\ldots, n-1$ to each of the other vertices of $G$. Since this coloring is not distinguishing, there is a non-identity automorphism $\beta$ of $G$ which swaps the vertices $u$ and $v$ and fixes all the other vertices. Thus  $N(u)\setminus\{v\}=N(v)\setminus\{u\}$, where $N(u)$ is the set of neighbors of $u$. Using a similar coloring argument, one deduces that this equality holds for any pair of vertices of $G$,  which implies that $G$ is either the complete or the empty graph.
	\end{proof}

	As we noted in the introduction, the problem of distinguishing colorings of infinite graphs has been studied in many interesting research works. It is, therefore, an interesting problem to consider the distinguishing threshold for infinite graphs. We conclude this section with the following theorem which states that in order to guarantee that any coloring of an infinite graph breaks its non-trivial symmetries, one needs infinitely many colors. 
	
	\begin{theorem}
		Let $G$ be an infinite graph. Then either $\theta(G)=1$ or $\theta(G)=\infty$.
	\end{theorem}
	\begin{proof}
		If $G$ is asymmetric, then $\theta (G)=1$. Hence we assume there is a non-trivial $\alpha\in\aut (G)$. If $\alpha$ has finite order, all its cycles are finite, so it must have infinitely many cycles. If $\alpha$ has infinite order then $\alpha^m$ has at least $m$ cycles for any $m\in \mathbb{N}$. Thus there is no finite bound on $\theta(G)$.
	\end{proof}

	\section{Johnson Scheme}\label{johnson}
	In this section we determine the distinguishing threshold of the graphs in the  Johnson scheme. 
	We assume  $n$ and $k$ are integers such  that $n\geq 2k\geq 2$, and  the set $\{1,\ldots,n\}$ is denoted by $[n]$. For any $i=1,\ldots, k$,  the graph $J(n,k,i)$ is defined to be the graph whose vertex set is the set of all $k$-subsets of $[n]$ and in which two vertices $A$ and $B$ are adjacent if  $|A\cap B|=k-i$. It can be shown that the set $\mathcal{A}=\{A_0,A_1,\ldots, A_k\}$, where $A_0$ is the identity matrix of order ${n\choose k}$ and, for any $i=1,\ldots,k$,  $A_i$ is the adjacency matrix of $J(n,k,k-i)$, constitutes  an association scheme (we refer the reader to~\cite{eiichi1984algebraic} for detailed studies on association schemes). This scheme is called the Johnson scheme, denoted by $J(n,k)$, and the classes $J(n,k,i)$, $1\leq i\leq k$, are called the generalized Johnson graphs. 
	The special cases of $J(n,k,k)$ and $J(n,k,1)$ are called the \emph{Kneser graph}, denoted by $K(n,k)$, and the \emph{Johnson graph}, respectively.
	
	The natural action of the symmetric group $\sym(n)$ on $[n]$, obviously preserves the adjacency-nonadjacency  relations in $J(n,k,i)$, for any $i=1,\ldots, k$. Therefore, $\sym(n)$ is isomorphic to a subgroup of  $\aut (J(n,k,i))$. The automorphism groups of the so-called ``merged Johnson graphs'' have been evaluated by Jones~\cite{jones2005automorphisms}. The merged Johnson graphs are, indeed, the unions of some graphs in the Johnson scheme. We re-state~\cite[Theorem~2]{jones2005automorphisms} to suit our case, where we study the individual graphs in the scheme, i.e. the generalized Johnson graphs. As in~\cite{jones2005automorphisms}, we will use the notation $e=\frac{1}{2}{n\choose k}$ and that given two groups $\Gamma_1$ and $\Gamma_2$, the groups $\Gamma_1 : \Gamma_2$ and  $\Gamma_1\wr \Gamma_2$ are their ``semi-direct product'' and ``wreath product'', respectively.
	
	\begin{theorem}\textnormal{\cite{jones2005automorphisms}}
		\label{jones}
		Assume that $2\leq k\leq n/2$.
		\begin{enumerate}[(a)]
			\item If $2\leq k<\frac{n-1}{2}$, then $\aut(J(n,k,i))\cong \sym(n)$, for each $i=1,\ldots,k$.
			\item If $k=\frac{n-1}{2}$ and  $i\neq \frac{k+1}{2}$, then $\aut(J(n,k,i))\cong \sym(n)$.
			\item If $k=\frac{n-1}{2}$ and  $i= \frac{k+1}{2}$, then $\aut(J(n,k,i))\cong \sym(n+1)$.
			\item If $k=\frac{n}{2}$ and  $k=2$, then $\aut(J(n,k,i))\cong \sym(2)\wr\sym(3)$,  for each $i=1,2$.
			\item If $k=\frac{n}{2}$, $k>2$, $i<k$ and $i\neq \frac{k}{2}$, then $\aut(J(n,k,i))\cong \sym(2)\times\sym(n)$.
			\item If $k=\frac{n}{2}$, $k>2$ and $i =  \frac{k}{2}$, then $\aut(J(n,k,i))\cong \sym(2)^e : \sym(n)$.
			\item If $k=\frac{n}{2}$, $k>2$ and $i=k$, then $\aut(J(n,k,i))\cong \sym(2)\wr\sym(e)$.\qed
		\end{enumerate}
	\end{theorem}
	
	Note that in the case  (d), the Johnson scheme $J(n,k)=J(4,2)$ has only two complementary graphs on $6$ vertices, where $J(4,2,2)=K(4,2)$ consists of $3$ copies of $K_2$; hence the automorphism groups of $J(4,2,1)$ and $J(4,2,2)$ are isomorphic to $ \sym(2)\wr\sym(3)$. Similarly, in the case (g), the generalized Johnson graph $J(2k,k,k)$, i.e. the Kneser graph $K(2k,k)$, consists of $e$ copies of $K_2$ resulting in the mentioned automorphism group.

	\begin{remark}\label{sym(n+1)_action}
		According to~\cite{jones2005automorphisms}, the rather unusual action of $\aut(J(n,k,i))\cong \sym(n+1)$ mentioned in part (c) of Theorem~\ref{jones}  is as follows. We add an external object  $\infty$ to $N=[n]$ to obtain  a new set $N^\ast=\{1,\ldots, n,\infty\}$ and consider the natural action of $\sym(n+1)$ on $N^\ast$. If $\sigma\in \sym(n+1)$, in order to find the image of a vertex $X\in V(J(n,k,i))$ under the corresponding automorphism $\tilde{\sigma}$ of $\sigma$ in $\aut(J(n,k,i))$, we construct the equipartition $(P_1,P_2)=(X\cup\{\infty\} , \overline{X\cup\{\infty\}}   )$ of $N^\ast$. Then, we make $\sigma$ act naturally on the pair $(P_1,P_2)$ to get $(P_1',P_2')$. The image $\tilde{\sigma} (X)$ is, then, $P_j'\setminus\infty$, where $\infty\in P_j'$. Because of the adjacency condition on the generalized Johnson graph $J(n,k,i)$, this is indeed an automorphism. It is clear that    the stabilizer of $\infty$ in $\sym(n+1)$, which is equal to $\sym(n)$, acts naturally (as mentioned before Theorem~\ref{jones}) on the vertices of $J(n,k,i)$.   Consider, for example, the case of $J(7,3,2)$ whose parameters satisfy the conditions in part (c). If $\sigma=(1\; 2\;\infty)(3\;4\;5)(6\;7) \in \sym(8)$ and $X=\{1,2,3\}, Y=\{3,4,5\}\in V(J(7,3,2))$, then  $\tilde{\sigma}(X)=\{2,4,1\}$ and $\tilde{\sigma}(Y)=\{2,7,6\}$.
	\end{remark}

	Using Theorem~\ref{jones},  one can evaluate the distinguishing number and the distinguishing threshold of the generalized Johnson graphs. We first note that the Kneser graph $K(n,1)=J(n,1,1)$ is the complete graph $K_n$; hence $\D(K(n,1))=\vert K(n,1)\vert =n$. 
	In addition, the Kneser graph $K(5,2)=J(5,2,2)$ is isomorphic to the Petersen graph. Albertson and Collins proved~\cite{albertson1996symmetry} that the distinguishing number of the Petersen graph $K(5,2)$ is equal to  $3$. Surprisingly, for $k\geq 2$ the Petersen graph is the only Kneser graph which is not $2$-distinguishable. In fact, Albertson and Boutin proved~\cite{albertson2007using} that for any $n\neq 5$ and $k\geq 2$, we have $\D(K(n,k))=2$. Furthermore, based on Theorem~\ref{jones},  Kim et al. in~\cite{kim2011distinguishing} evaluated the distinguishing number of all merged Johnson graphs which generalized the above results. Similar to Theorem~\ref{jones}, we rephrase their results, in terms of single classes of the Johnson scheme, as follows.
	
	\begin{theorem}\textnormal{\cite{kim2011distinguishing}}\label{d_of_generalized_johnson}
		Assume that $2\leq k\leq n/2$.
		\begin{enumerate}[(a)]
			\item If $n=5$ and $k=2$, then  $\D(J(n,k,1))=\D(J(n,k,2))=3$.
			\item If $n\neq 5$ and  $2\leq k<\frac{n}{2}$, then $\D(J(n,k,i))=2$, for each $i=1,\ldots,k$.
			\item  If  $ k=\frac{n}{2}$ and $i\notin\{ \frac{k}{2} , k\}$, then $\D(J(n,k,i))=2$.
			\item If  $ k=\frac{n}{2}$ and  $i=\frac{k}{2}$, then $\D(J(n,k,i))=3$.
			\item  If   $ k=\frac{n}{2}$ and  $i=k$, then $\D(J(n,k,i))=\lceil  \frac{1+\sqrt{1+8e}}{2}\rceil$.\qed
		\end{enumerate}
	\end{theorem}

	We now turn our attention to the problem of determining the distinguishing  threshold of the graphs in the Johnson scheme. As mentioned above, $J(n,1,1)$ consists only of the complete graph $K_n$; hence $\theta(J(n,1,1))=n$. Thus, we consider only the case where $k\geq 2$.

	\begin{theorem}\label{theta_of_johnson}
		Assume that $2\leq k\leq n/2$.
		\begin{enumerate}[(a)]
			\item If $2\leq k< n/2$, then $\theta(J(n,k,i))={n\choose k} - {n-2\choose k-1}+1$, for each $i=1,\ldots,k$.
			\item  If  $ k=\frac{n}{2}$ and $i\in\{ \frac{k}{2} , k\}$, then $\theta(J(n,k,i))={n\choose k}$.
			\item If  $ k=\frac{n}{2}$ and $i\notin\{ \frac{k}{2} , k\}$, then $\theta(J(n,k,i))={n\choose k} - {n-2\choose k-1}+1$.
		\end{enumerate}
	\end{theorem}
	\begin{proof}
		We denote $J(n,k,i)$ by $J$ and the number of vertices of $J$ by $\nu$.  In order to show (a), first we note, according to Theorem~\ref{jones},  that  $\aut(J)\cong \sym(n)$ or $\sym(n+1)$. Consider the automorphism $\tilde{\alpha}\in \aut(J)$ defined as 
		\[\tilde{\alpha}=(X_1\cup \{1\} \,,\, X_1\cup \{2\})\;  (X_2\cup \{1\}  \,,\, X_2\cup \{2\})\cdots (X_r\cup \{1\}  \,,\, X_r\cup \{2\}),\]
		where $X_1,\ldots, X_r$ are all the $(k-1)$-subsets of $\{3,4,\ldots,n\}$ and, hence, 
		\[r={n-2\choose k-1}.\]
		Note that $\tilde{\alpha}$ is the image of the transposition $\alpha=(1, 2)\in \sym(n)$ or $\alpha=(1,2)\in \sym(n+1)$ under action of  $ \sym(n) $ or $ \sym(n+1)$, respectively, on $J$. The number of cycles in $\tilde{\alpha}$ is 
		\begin{equation}\label{number_of_cycles_of_alpha_tilde}
			c(\tilde{\alpha})= \nu- r= {n\choose k} - {n-2\choose k-1}.
		\end{equation}
		Hence, it suffices to show that $c(\tilde{\alpha})$ is the maximum in $\aut(J)$. Assume first that $\aut(J)\cong \sym(n)$ and  let $\beta \in \sym(n)$ be any non-identity permutation and $\tilde{\beta}$ be its image in $\aut(J)$. Let $M$ be the set of moved points of $\beta$ acting on $[n]$  and that $m=|M|$; hence $1<m\leq n$. A vertex $X\in V(J)$ is fixed under the action of $\beta$ on $V(J)$ if and only if 
		$M\cap X=\emptyset$ or $M\subseteq X$. Therefore, the number of vertices in $V(J)$ fixed by  $\tilde{\beta}$ is
		$x={n-m\choose k-m} + {n-m\choose k}$,
		and the number of  vertices moved by $\tilde{\beta}$ is ${n\choose k} -x$.
		It then follows that 
		\[
		c(\tilde{\beta})\leq x+ \frac{1}{2} \left[{n\choose k} -x\right] = \frac{1}{2}  \left[{n\choose k}  + x\right].
		\]
		On the other hand, 
		\[
		x\leq {n-2\choose k-2} + {n-2\choose k}.
		\]
		Hence
		\[c(\tilde{\beta})\leq \frac{1}{2}  \left[{n\choose k}  +  {n-2\choose k-2} + {n-2\choose k}  \right]= c(\tilde{\alpha}),\]
		which proves the claim. 
		
		Now, assume that $\aut(J)\cong \sym(n+1)$. Suppose $\sigma\in\sym(n+1)$.  According to Remark~\ref{sym(n+1)_action}, if $\sigma$ fixes~$\infty$, then similar to the previous  case, we have $c(\tilde{\sigma})\leq c(\tilde{\alpha})$. Hence, we assume that $\sigma(\infty)\neq \infty$. A vertex $X$ of $J$ is fixed by the  automorphism $\tilde{\sigma}$, if and only if $X$ is set-wise stabilized under the action of $\sigma$. Thus, in order to have the most number of fixed vertices, $\sigma$ must have   a largest fixed set. We conclude that $c(\tilde{\sigma})$ is maximized when $\sigma$ is a transposition. Without loss of generality, we assume that  $\sigma=(1\; \infty)$. Then it is easy to see that
		\[\tilde{\sigma}= (X_1, \overline{X_1\cup\{1\}}) (X_2, \overline{X_2\cup\{1\}})\cdots (X_s, \overline{X_s\cup\{1\}}), \]
		where $X_i$ are all the $k$-subsets of $\{2,3,\ldots,n\}$. Hence
		\[c(\tilde{\sigma}) = {n\choose k} - \frac{1}{2}{n-1\choose k} = {n\choose k} - \frac{k}{n-1}{n-1\choose k}=  {n\choose k} - {n-2\choose k-1} = c(\tilde{\alpha}), \]
		which complete the proof of (a).
		
		Part (b) follows from the fact that the map  which swaps the two vertices  $\{1,\ldots,k\}$ and $\{k+1,\ldots,n\}$, and fixes all the other vertices of $J$, is indeed an automorphism of $J$ having $\nu-1$ cycles.
		
		To show part (c), the map  $\tilde{\alpha}$ defined  in part (a) is again in  $\aut(J)$, as well. Hence it suffices to show that it, again, has the largest number of cycles in $\aut(J)$. According to Theorem~\ref{jones}, $\aut(J)\cong \sym(2)\times \sym(n)$. In fact, the complementary map $\delta$ is an automorphism of $J$ (of order two) and any automorphism $\tilde{\gamma}$ of $J$ is of the form $\tilde{\gamma}=\delta^j \tilde{\beta}=\tilde{\beta} \delta^j$, where $j=0$ or $1$ and $\tilde{\beta}$ is the image of a permutation $\beta\in \sym(n)$ under the natural action of $\sym(n)$ on $[n]$. This implies that  $c(\tilde{\gamma})\leq c(\tilde{\alpha})$ and, hence, the result follows.
	\end{proof}
	
	In~\cite{ahmadi2020number}, the distinguishing threshold of the Kneser graphs $K(n,2)$ has been determined. We point out that Theorem~\ref{theta_of_johnson} generalizes this result to all Kneser graphs $K(n,k)$.

	\section{Conclusion and future work}\label{conclusion}
	The problem of finding the distinguishing threshold of a graph seems an interesting one and in this paper we considered it in more depth. Along with studying graphs with small thresholds, we established a strong connection between the threshold and the cycle structure of the automorphisms of the graph and, using this approach, we could completely evaluate the thresholds of the generalized Johnson graphs.
	
	We conclude the paper with considering  a similar problem for Cayley graphs $\mathrm{Cay}(\Gamma,S)$. While the distinguishing number and the distinguishing index of Cayley graphs are studied by Alikhani and Soltani \cite{Alikhani-Cayley2021}, we propose the following  problem in which the knowledge of the structures of the automorphism groups can play a key role. 
	
	\bigskip
	\noindent \textbf{Problem.} Evaluate the distinguishing threshold of Cayley graphs.
	\bigskip
	
	\noindent 
	We recall that for any $g\in \Gamma$, the action of $\Gamma$ on itself by left multiplication by $g$ naturally induces an automorphism of  $\mathrm{Cay}(\Gamma,S)$ (which, in turn, shows that Cayley graphs are vertex-transitive). Therefore, in this case, one already knows that $\Gamma$ is isomorphic to a subgroup of $\aut(\mathrm{Cay}(\Gamma,S))$. Indeed  the family of Cayley  graphs is another family of graphs whose automorphism groups have been studied extensively; see for example~\cite{jajcay2000structure} by Jajcay.  As mentioned in Section~\ref{threshold}, we consider only connected Cayley graphs, i.e. the graphs $\mathrm{Cay}(\Gamma,S)$ in which $S$ generates the group $\Gamma$. 
	Based on examples, the problem looks quite challenging.  For instance, complete graphs are Cayley graphs with trivial thresholds while the Cayley graphs $\mathrm{Cay}(\mathbb{Z}_n, \{a,-a\})$, where $0\neq a$ is relatively prime to $n$, are cycles whose thresholds are $\lfloor\frac{n}{2}\rfloor+2$ for $n\geq 3$, which is not trivial. As another example, in which the group is abelian but not  cyclic, consider $G_6=\mathrm{Cay}(\mathbb{Z}_2\oplus \mathbb{Z}_3 , \{(1,0), (0,1), (0,2)\})$; see Figure~\ref{G_6}.
	
	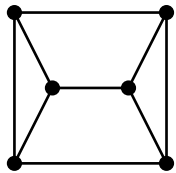
\begin{figure}[h!]
		\centering
		\begin{center}
			
			\begin{tikzpicture}  [scale=0.5]
				
				\tikzstyle{every path}=[line width=1pt]
				
				\newdimen\ms
				\ms=0.1cm
				\tikzstyle{s1}=[color=black,fill,rectangle,inner sep=3]
				\tikzstyle{c1}=[color=black,fill,circle,inner sep={\ms/8},minimum size=2*\ms]
				
				
				\coordinate (a1) at  (0,0);
				\coordinate (a2) at (4,0);
				\coordinate (a3) at (4,4);
				\coordinate (a4) at (0,4);
				\coordinate (a5) at (1,2);
				\coordinate (a6) at (3,2);
				
				\draw [color=black] (a1) -- (a2);
				\draw [color=black] (a2) -- (a3);
				\draw [color=black] (a3) -- (a4);
				\draw [color=black] (a4) -- (a1);
				\draw [color=black] (a1) -- (a5);
				\draw [color=black] (a4) -- (a5);
				\draw [color=black] (a5) -- (a6);
				\draw [color=black] (a6) -- (a2);
				\draw [color=black] (a6) -- (a3);
				
				\draw (a1) coordinate[c1];
				\draw (a2) coordinate[c1];
				\draw (a5) coordinate[c1];
				\draw (a6) coordinate[c1];
				\draw (a3) coordinate[c1];
				\draw (a4) coordinate[c1];
			\end{tikzpicture}
			
		\end{center}
		\caption{The graph $G_6$.}
		\label{G_6}
	\end{figure}
	
	It is not hard to see that the map $\alpha$ which swaps the    vertices $(0,1), \, (0,2)$, swaps  the   vertices $(1,1),\, (1,2)$, and fixes the other vertices, is an automorphism of $G_6$ with $c(\alpha)=4$ which is maximum. Thus, we have $\theta(G_6)=5$ which displays another non-trivial example. Further, as a non-abelian non-trivial example, one can consider  $G_{24}=\mathrm{Cay}(S_4, \{(1,2) , (2,3), (2,4)\})$ on the symmetric group $S_4$ which is a connected cubic bipartite graph. With an easy computer search, we observe that there is an $\alpha\in \aut(G_{24})$ that is the product of $8$ transpositions and, hence, $c(\alpha)=16$  which is maximum. This shows that $\theta(G_{24})=17$, far smaller than  the number of vertices.

	\section*{Acknowledgment}
	The  authors would like to express their special thanks to Amir Mohammad Ghazanfari for his enlightening comments during the research work.

	\bibliographystyle{plain}
	\bibliography{bibliography}
	
\end{document}